\documentclass[11pt]{amsart}
\usepackage[utf8]{inputenc}
\usepackage{amsfonts, amssymb}
\usepackage{amssymb,amsmath,amscd,euscript,verbatim,array}
\usepackage{geometry}
\usepackage{anysize}
\usepackage{fancyhdr}
\usepackage{indentfirst}
\usepackage{graphicx}
\usepackage{ifpdf}
\usepackage{marginnote}
\usepackage{enumitem}
\usepackage{mathrsfs}
\usepackage{bm}
\usepackage{xcolor}
\setlist[itemize]{leftmargin=2em}

\newcommand{\R}{\mathbb{R}}

\newcommand{\D}[2]{\displaystyle \frac{\partial #1}{\partial #2}}

\newtheorem{MainTheorem}{Theorem}

\newtheorem{Theorem}{Theorem}
\newtheorem{Definition}{Definition}[section]
\newtheorem{Proposition}[Definition]{Proposition}
\newtheorem{Corollary}[Definition]{Corollary}
\newtheorem{Lemma}[Definition]{Lemma}

\theoremstyle{definition}
\newtheorem{Remark}[Definition]{Remark}

\newcommand{\bthm}{\begin{Theorem}}
	\newcommand{\ethm}{\end{Theorem}}
\newcommand{\bpr}{\begin{Proposition}}
	\newcommand{\epr}{\end{Proposition}}
\newcommand{\blm}{\begin{Lemma}}
	\newcommand{\elm}{\end{Lemma}}
\newcommand{\bex}{\begin{Exercise}}
	\newcommand{\eex}{\end{Exercise}}
\newcommand{\be}{\begin{equation}}
	\newcommand{\ee}{\end{equation}}
\newcommand{\beal}{\begin{aligned}}
	\newcommand{\enal}{\end{aligned}}
\newcommand{\brm}{\begin{Remark}}
	\newcommand{\erm}{\end{Remark}}

\definecolor{darkgreen}{HTML}{006400}
	
	\usepackage{hyperref}
	\usepackage{color}

\newcommand{\rev}[1]{\textcolor{black}{#1}}
\RequirePackage{url}

\begin{document}

	\title[Quantitative Persistence and Regularity of Invariant Graphs]{Quantitative Persistence and Regularity  of Invariant Graphs for  Dissipative Twist Maps}
	
	\author{Deyu Kong}
	\address{School of Mathematics and Statistics, Beijing Institute of Technology, Beijing 100081, China}
	\email{dykong@bit.edu.cn}
	
	\author{Jessica Elisa Massetti}
	\address{Department of Mathematics, University of Rome Tor Vergata, Via della ricerca scientifica 1, 00133 Rome, Italy}
	\email{massetti@axp.mat.uniroma2.it}

	\author{Lin Wang}
	\address{School of Mathematics and Statistics, Beijing Institute of Technology, Beijing 100081, China}
	\email{lwang@bit.edu.cn}
	
	\subjclass[2010]{Primary 37J40; Secondary 37E40}
	
	\keywords{Dissipative Twist Maps, Invariant Graphs, Lagrangian Submanifolds}

	\begin{abstract}
		In this paper we study the robustness of invariant Lagrangian graphs for dissipative twist maps on  $\mathbb T^d\times\mathbb R^d$ in the $C^r$ category, for any $r\ge 1$, and show their persistence as invariant global attractors under quantitative sharp conditions involving both the size of dissipation parameter and the perturbation at stake. Note that when $d>1$ the geometric Lagrangian property of the invariant manifold may be destroyed, so that the graph may persist but fail to be Lagrangian. In general, additional dynamical requirements on the twist maps come into play if we intend to preserve also the Lagrangian property. In this line, we first study their persistence for general perturbations, then we restrict to the class of exact conformally symplectic twist maps and provide quantitative sharp conditions for the persistence of invariant exact Lagrangian graphs, without assuming any condition on the tangential dynamics of the manifold. The persistence result relies on a combination of a suitable multidimensional Herman--Mather
formula proved in this setting and contraction estimates that provide sharp bounds on the $C^1$-norm of the perturbation to get a unique, globally attracting $C^r$ graph. In dimension $d=1$, by further assuming that that the circle map induced by the graph has irrational
rotation number, we prove that every $C^1$ invariant graph of a $C^k$ map is $C^k$
for $k=2,3,\ldots,\infty$. Analytic maps with Lipschitz, non-$C^1$ invariant
graphs at every prescribed irrational rotation number show that
this initial regularity assumption is essential.
\end{abstract}

	\maketitle
	
	{\footnotesize\setlength{\baselineskip}{10pt}\vspace{-\baselineskip}\tableofcontents}
	
	\newpage

\section{Introduction and main results}

We study invariant graphs for dissipative twist maps on
$\mathbb T^d\times\mathbb R^d$, $d\ge1$, whose lifts have the form
\begin{equation}\label{eq:dissipative_map}
 F_{\Lambda,\alpha}^{\phi}(x,y)
 =\bigl(x+\alpha_1+\Lambda_1y+\phi_1(x),\,
         \alpha_2+\Lambda_2y+\phi_2(x)\bigr).
\end{equation}
Here $\alpha=(\alpha_1,\alpha_2)\in\mathbb R^{2d}$,
$\Lambda_1,\Lambda_2\in GL_d(\mathbb R)$, and the perturbations
$\phi_1,\phi_2:\mathbb R^d\to\mathbb R^d$ are $\mathbb Z^d$-periodic.
We assume that $F_{\Lambda,\alpha}^{\phi}$ is a global diffeomorphism
and that its projection $f_{\Lambda,\alpha}^{\phi}$ on the cylinder
is isotopic to the identity and satisfies
\begin{equation}\label{dissp}
 0<\det DF_{\Lambda,\alpha}^{\phi}(x,y)<1.
\end{equation}
Periodicity gives the lift relation
$F(x+k,y)=F(x,y)+(k,0)$ for $k\in\mathbb Z^d$.
The full conventions are given in Definition~\ref{DTM}.
\rev{Our main contraction assumption is
\begin{equation}\label{assumption M}
 M:=\Lambda_1\Lambda_2\Lambda_1^{-1},\qquad \|M\|<1,
\end{equation}
where $\|\cdot\|$ stands for the operator norm induced by the standard euclidean metric on $\R^d$.}
We seek explicit conditions under which an invariant graph exists,
determine its regularity, and ask when it is Lagrangian. \rev{We stress the fact that in dimension greater than $2$, the Lagrangian property is not automatically preserved together with the invariant graph and there exist cases where the invariant graph persists, but its Lagrangian property is destroyed by the perturbation (see Section \ref{sec:non_lagrangian}). This geometric property is kept only under further conditions on the admissible perturbations or on under dynamical hypothesis of the system at stake (see Section \ref{intro conformally}). For this reason we shall treat the two matters separately and distinguish between a more general result on existence of $C^k$-invariant graph and one of invariant Lagrangian graph, where we shall restrict to the class of exact conformally sympletic twist map, the geometric property being more involved.}

The study of attracting invariant sets for dissipative annulus maps
goes back to Birkhoff~\cite{Bir32}. Later work established periodic
orbits for dissipative twist maps~\cite{Cas87} and quasi-periodic
orbits in Birkhoff attractors~\cite{LeC86}. Arnaud, Humili\`ere,
and Viterbo~\cite{AHV} extended the notion of a Birkhoff attractor
to higher-dimensional conformally symplectic dynamics. \\
\rev{ In the present paper, we study the competition between size of the perturbation and strength of attraction, focusing on the study of the regime in which an invariant graph persists, the
geometric structure it retains, and the assumptions needed for its
regularity. }

\subsection{Persistence and the Herman--Mather formula}

When $\phi_1=\phi_2=0$, the integrable map is
\begin{equation}\label{I-G,Gamma0}
 F_{\Lambda,\alpha}(x,y)
 =\bigl(x+\alpha_1+\Lambda_1y,\,\alpha_2+\Lambda_2y\bigr).
\end{equation}
The contraction assumption implies that $I_d-\Lambda_2$ is invertible.
The unique continuous periodic invariant graph is
\[
 \Gamma_0:=\mathbb R^d\times
 \bigl\{(I_d-\Lambda_2)^{-1}\alpha_2\bigr\},
\]
and its projection on the cylinder attracts every orbit. The induced
base map is the rigid translation by
\[
 \omega_0:=\alpha_1+\Lambda_1(I_d-\Lambda_2)^{-1}\alpha_2.
\]
On the one hand, the normally hyperbolic invariant manifold theorem~\cite{HPS77}
ensures the persistence under sufficiently small perturbations, on the other one, our purpose here is to give
explicit bounds that relate the perturbation size to the contraction
and the desired regularity of the graph.
\rev{We stress the fact that we are not concerned with the control on the tangential dynamics of invariant manifold for which, in the frame of
conformally symplectic systems, a KAM-type approach prove that for certain parameters-dependent families of diffeomorphisms or vector fields, invariant tori
with prescribed Diophantine dynamics persist under suitable nondegeneracy
conditions on the system and appropriate choice of parameters (see for example \cite{CCD13,EFK:2015,Massetti:ETDS,CCD20}.}
\rev{On the other hand, one can take full advantage of the Diophantine property of the rotation vector and use normal form techniques in different directions, in order to pull-back the system in a suitable form and average out up to a certain order the angle-dependence that obstruct the manifold to be invariant, and finally combine a graph transform argument to recover the persistence as in the classic normally hyperbolic context. See for instance
Massetti~\cite{Mas23} and references therein,
for the study of stability of attracting circles for general real analytic perturbations of dissipative twist maps.} \\
The following result applies to every drift $\alpha$
and requires no arithmetic condition on $\omega_0$ or assumption on
the dynamics of the perturbed graph. Norms and nonintegral regularity
are specified in Section~\ref{sec:notation}.

\begin{MainTheorem}[Persistence of an invariant graph]
\label{thm:graph_persistence}
Let $1\le r<\infty$, and let $F_{\Lambda,\alpha}^{\phi}$ be a $C^r$
dissipative twist map as above. Choose $L$ such that
\[
\begin{cases}
 0<L\le1-\|M\|^{1/2},&r=1,\\
 0<L<1-\|M\|^{1/(r+1)},&r>1.
\end{cases}
\]
Suppose that
\begin{equation}\label{eq:persistence_smallness}
\begin{aligned}
 &\|\phi_1\|_{C^1}
 +\frac{\|\Lambda_1\phi_2-M\phi_1\|_{C^1}}{1-L}
 \le\frac{L(1-L-\|M\|)}{1-L},\\
 &\|\Lambda_1\phi_2-M\phi_1\|_{C^1}
 \begin{cases}
 \le(1-L)^2-\|M\|,&r=1,\\
 <(1-L)^{r+1}-\|M\|,&r>1.
 \end{cases}
\end{aligned}
\end{equation}
Then $F_{\Lambda,\alpha}^{\phi}$ has a unique continuous periodic
invariant graph $\Gamma_\psi=\{(x,\psi(x)):x\in\mathbb R^d\}$, with
$\psi\in C^r(\mathbb T^d,\mathbb R^d)$. The projected graph
$\bar\Gamma_\psi$ is a global attractor: the distance of every orbit
to it tends to zero exponentially, uniformly for initial points in
bounded sets.
\end{MainTheorem}

The proof relies on two main ingredients. \rev{The first one is a generalization of the
Herman--Mather formula (see \cite{Herman:2018, Her83}) to the multidimensional dissipative setting, that enables us to translate the problem of persistence in terms of finding the ``fixed point" of an appropriate operator.\\ More specifically, in
Proposition~\ref{pro:Herman} we characterize continuous invariant graphs
by a functional equation for their induced base maps and recover
the graph explicitly from any homeomorphism satisfying that equation.}
This identifies graph invariance with a fixed-point problem for an
operator on base maps.

The second ingredient is to prove that in fact the fixed point equation is solvable, via a $C^r$-contraction fixed-point theorem in appropriate Banach space. See Theorem~\ref{thm:DSM_Cr_smoothness} in
Section~\ref{sec:persistence}. The Banach contraction
principle~\cite[Theorem~5.7]{Bre10} gives a Lipschitz fixed point;
convergence of the differentiated iterates gives its $C^r$ regularity.
At the $C^1$ endpoint, a strict nonlinear estimate for matrix inverses
replaces a uniform linear contraction of the derivative equation.
Normal hyperbolicity is not asserted at this endpoint. For $r>1$,
the strict inequalities in \eqref{eq:persistence_smallness} give
the bounds needed for higher regularity. The same construction proves
global attraction and uniqueness among continuous invariant graphs.

\subsection{Lagrangian graphs and exact conformally symplectic maps}\label{intro conformally}

In dimension one, every $C^1$ graph is Lagrangian. For $d\ge2$,
arbitrarily small analytic perturbations can preserve the invariant
graph while destroying its Lagrangian property; an explicit example
is given in Section~\ref{sec:non_lagrangian}. This leads us to ask
which geometric assumptions ensure that the persistent graph remains
Lagrangian. \rev{ A good class to work with, is the one of exact conformally symplectic (ECS) family
\begin{equation}\label{eq:conformal_symplectic_twist_map}
 F^{\Phi}_{\lambda,\alpha}(x,y)
 =\bigl(x+\alpha+\lambda y+D\Phi(x),\,
         \lambda y+D\Phi(x)\bigr),
 \qquad 0<\lambda<1,
\end{equation}
where $\Phi$ is a periodic potential, to which we shall restrict. The above map is exact conformally
symplectic with conformal factor $\lambda$, meaning that ${F^\Phi_{\lambda,\alpha}}^*\Omega = \lambda \Omega$, where $\Omega$ is the standard symplectic $2$-form; for the precise notion see
Definition~\ref{def:Conformally Symplectic Twist Map}.}
Mar\`o and Sorrentino~\cite{MS17} developed an Aubry--Mather theory
for conformally symplectic systems, related to discounted
Hamilton--Jacobi equations. Arnaud and F\'ejoz~\cite{AF24} gave
criteria for invariant submanifolds to be Lagrangian involving
dimension and entropy. \rev{Here we exploit conditions that do not depend on the dynamics carried by the graph and give explicit bounds that guarantee that the family
\eqref{eq:conformal_symplectic_twist_map} admit an invariant exact Lagrangian graph.}  Let us set
\begin{equation}\label{eq:eta_r}
 \eta_r(\lambda):=
 (1-\lambda^{1/(r+1)})(1-\lambda^{r/(r+1)}),\qquad 1\le r<\infty.
\end{equation}

\begin{MainTheorem}[Persistence of exact Lagrangian graphs]
\label{thm:main_lagrangian}
Let $d\ge1$, $0<\lambda<1$, $\alpha\in\mathbb R^d$, and
$1\le r<\infty$. Suppose that
$\Phi\in C^{r+1}(\mathbb T^d,\mathbb R)$ satisfies
\[
 \begin{cases}
 \|D\Phi\|_{C^1}\le\eta_1(\lambda)=(1-\sqrt\lambda)^2,&r=1,\\
 \|D\Phi\|_{C^1}<\eta_r(\lambda),&r>1.
 \end{cases}
\]
Then $F^{\Phi}_{\lambda,\alpha}$ has a unique continuous periodic
invariant graph. This graph is of class $C^r$ and is exact
Lagrangian, and its projection on the cylinder is a global attractor.
\end{MainTheorem}

The proof first applies Theorem~\ref{thm:graph_persistence} to
construct the $C^1$ graph. The conformal pullback identity for the
symplectic form then proves that its restriction to the graph vanishes;
at the endpoint, this requires an analysis of equality in the norm
estimates. The corresponding identity for the Liouville form gives
exactness. Higher regularity follows by optimizing the parameter $L$
in the contraction estimates. These arguments are carried out
separately in Theorems~\ref{thm:Lagrangian} and~\ref{thm:r-map}.

For $d=1$, the $C^1$ persistence conclusion agrees with  \cite[Theorem~1]{HLW26}. {The bound is optimal as a uniform
condition to get $C^1$ regularity , by \cite[Theorem~2]{HLW26}; the same
obstruction extends to higher dimensions by taking products with
integrable maps.  Proposition~\ref{prop:C2_endpoint} and Remark \ref{rem:higher_endpoint_counterexamples} show that the $C^1$ threshold in
Theorem~\ref{thm:Lagrangian} does not extend to the corresponding
$C^r$ threshold for $r>1$. The comparison with the destruction results for invariant
Lagrangian graphs provided in ~\cite{SW26} is given in
Section~\ref{sec:ecs_sharpness}.

\subsection{Regularity in dimension one}

The upper bounds $\eta_r(\lambda)$ introduced above tend to zero as $r\to\infty$, so they
do not give $C^\infty$ regularity for a fixed nonzero perturbation.
In dimension one, the rotation number is somehow related to the
regularity and can be used to derive information about. The conservative theory of invariant curves was developed
by Herman, Salamon, and Katznelson--Ornstein~\cite{Her83,Her86,Sa,KO}.
For dissipative maps, we ask whether irrational rotation number allows
an invariant graph to inherit the regularity of the ambient map,
and what regularity must initially be assumed of the graph.
Hall's smooth annulus maps with attracting Lipschitz circles carrying
Denjoy dynamics~\cite{Hall83} also motivate the latter question.

\begin{MainTheorem}[Regularity and the $C^1$ hypothesis]
\label{thm:main_regularity}
In dimension $d=1$, the following statements hold.
\begin{enumerate}[label=\textup{(\roman*)},leftmargin=2em]
\item Let $k\ge2$ be an integer, or let $k=\infty$, and let
$F_{\Lambda,\alpha}^{\phi}$ be a $C^k$ dissipative twist map
\eqref{eq:dissipative_map} satisfying \eqref{dissp}.
If it preserves a graph defined by $\psi\in C^1(\mathbb T,\mathbb R)$
and the induced circle map has irrational rotation number, then
$\psi\in C^k(\mathbb T,\mathbb R)$.

\item For every $\lambda\in(0,1)$ and
$\omega\in\mathbb R\setminus\mathbb Q$, there are
$t_*\in(0,1+\lambda)$ and $\alpha_*\in\mathbb R$ such that the
real-analytic ECS map \eqref{eq:conformal_symplectic_twist_map} with
\[
 \Phi(x)=-\frac{t_*}{4\pi^2}\cos(2\pi x),\qquad \alpha=\alpha_*,
\]
preserves a graph defined by
$\psi_*\in C^{0,1}(\mathbb T,\mathbb R)\setminus C^1(\mathbb T,\mathbb R)$.
The induced degree-one lift has rotation number $\omega$.
\end{enumerate}
\end{MainTheorem}

The first part requires neither a Diophantine condition nor a
smallness assumption. The induced $C^1$ circle diffeomorphism is
uniquely ergodic~\cite{Her79}. The uniform ergodic
theorem~\cite[Theorem~6.19]{Wal82}, applied to its logarithmic
derivative, gives uniform subexponential bounds on the derivatives of its
iterates. Combined with normal contraction, these estimates allow
successive applications of the $C^r$ section theorem~\cite{HPS77}.
The details are given in the proof of
Theorem~\ref{thm:C_infty_smoothness}.

The second part shows that the initial $C^1$ assumption cannot be
weakened to Lipschitz continuity, even for analytic maps. Its proof
fixes the irrational rotation number and takes a maximal amplitude
in a sinusoidal family. Uniform Lipschitz bounds give a limiting
invariant graph. If this graph were $C^1$, the first part would make
it smooth and the derivative estimates would imply normal
hyperbolicity. Persistence, together with a suitable change of drift,
would then produce a graph with the same rotation number at a larger
amplitude, a contradiction. Proposition~\ref{prop:irrational_Lip_not_C1}
contains the construction; Remark~\ref{rem:Hall_irrational_graphs}
explains its relation to Hall's questions and the remaining question
of wandering intervals.

\medskip
\noindent\textbf{Organization of the paper.}
Section~\ref{sec:notation} fixes the notation and definitions.
Section~\ref{sec:persistence} derives the Herman--Mather formula and
proves Theorem~\ref{thm:graph_persistence} through the contraction
theorem. Section~\ref{sec:non_lagrangian} gives the non-Lagrangian
example. Section~\ref{sec:ecs} proves
Theorem~\ref{thm:main_lagrangian} and examines the sharpness of its
bounds. Section~\ref{sec:regularity} proves both parts of
Theorem~\ref{thm:main_regularity}.

\medskip
\noindent\textbf{Acknowledgments.} J.E. Massetti has been supported by the INdAM GNAMPA Project ``Growth of Sobolev norms and energy cascades in dispersive Hamiltonian PDE " - CUP E53C25002010001 and acknowledges also the support of the Department of Excellence
grant MatMod@TOV (2023-27), awarded to the Department of Mathematics at University of Rome
Tor Vergata. D. Kong and L. Wang are partially supported by the National Natural Science Foundation of China (Grant No. 12671217 and 12122109).

\section{Functional setting, notations and definitions}\label{sec:notation}
	{We fix the notation, function spaces, and geometric definitions used below. These conventions apply for every $d\ge1$.}

	\subsection{Basic notations and global attractors}
	Let $\mathbb{T}^d = \mathbb{R}^d / \mathbb{Z}^d$ denote the standard $d$-dimensional torus. For any vector $x \in \mathbb{R}^d$, we write its coordinates as $x = (x_1, \dots, x_d)^T$. Throughout this paper, unless otherwise specified, $|\cdot|$ denotes the standard Euclidean norm on $\mathbb{R}^d$, $|v| = \big(\sum_{i=1}^d v_i^2\big)^{1/2}$, and $\|\cdot\|$ denotes the induced matrix operator norm on $\operatorname{Mat}_d(\mathbb{R})$,
	\[
	\|M\| := \sup_{v \neq 0} \frac{|Mv|}{|v|} = \max_{|v|=1} |Mv|.
	\]
	{We write $\mathrm{Id}$ for an identity map and $I_d$ for the $d\times d$ identity matrix.}

We write $\theta\in\mathbb T^d$ for a base point on the torus,
$x\in\mathbb R^d$ for a lift, and $y\in\mathbb R^d$ for the fiber
coordinate. A cylinder map is denoted by $f$ and its lift by $F$.
For the induced base map, we write $\bar g:\mathbb T^d\to\mathbb T^d$
and $g:\mathbb R^d\to\mathbb R^d$ for its degree-one lift.
A periodic graph function $\psi$ descends to $\bar\psi$, and its graph
$\Gamma$ projects to $\bar\Gamma$. Periodic coefficient functions and
their descended functions on the torus are identified when the
meaning is clear from the context.


We shall endow $\mathbb T^d$ with the standard $\operatorname{mod}-1$ metric and the cylinder with the product one:
\[
 d_{\mathbb T^d}(\theta,\theta')
 :=\min_{m\in\mathbb Z^d}|x-x'-m|,\qquad
 \bar d((\theta,y),(\theta',y'))
 :=d_{\mathbb T^d}(\theta,\theta')+|y-y'|,
\]
where $x,x'$ are arbitrary lifts of $\theta,\theta'$.
A set $\bar B$ is bounded if
$\sup_{\bar z,\bar z'\in\bar B}\bar d(\bar z,\bar z')<\infty$,
equivalently if its fiber coordinates are bounded.
For any pair of nonempty sets $\bar A,\bar B$, we set
\[
 \bar d(\bar A,\bar B)
 :=\inf_{\bar a\in\bar A,\bar b\in\bar B}\bar d(\bar a,\bar b).
\]

	\begin{Definition}[Global attractor, cf.~{\cite[p.~17]{Hal10}}]\label{def:Global Attractor}
		Let $f \colon \mathbb{T}^d \times \mathbb{R}^d \to \mathbb{T}^d \times \mathbb{R}^d$ be a continuous map on the cylinder equipped with the metric $\bar{d}$. An invariant set $\bar{\mathcal{A}} \subset \mathbb{T}^d \times \mathbb{R}^d$ is called a \emph{global attractor} for $f$ if it satisfies the following two conditions:
		\begin{enumerate}[label=(\roman*), font=\upshape]
			\item \emph{Maximality:} $\bar{\mathcal{A}}$ is a maximal compact invariant set; that is, $\bar{\mathcal{A}}$ is compact, invariant ($f(\bar{\mathcal{A}}) = \bar{\mathcal{A}}$), and contains every compact invariant set of $f$ on the cylinder.
			\item \emph{Global attraction:} $\bar{\mathcal{A}}$ attracts every bounded set in $\mathbb{T}^d \times \mathbb{R}^d$. Specifically, for any bounded set $\bar{B} \subset \mathbb{T}^d \times \mathbb{R}^d$ and any $\varepsilon > 0$, there exists an integer $n_0 = n_0(\varepsilon, \bar{B}) \ge 1$ such that $f^n(\bar{B}) \subset N_\varepsilon(\bar{\mathcal{A}})$ for all $n \ge n_0$, where $N_\varepsilon(\bar{\mathcal{A}}) := \{ \bar{z} \in \mathbb{T}^d \times \mathbb{R}^d \mid \bar{d}(\bar{z}, \bar{\mathcal{A}}) < \varepsilon \}$.
		\end{enumerate}
	\end{Definition}
	
	\begin{Remark}\label{rem:attractor_maximality}

Compactness, invariance, and attraction of every bounded set imply
maximality in Definition~\ref{def:Global Attractor}. Indeed, let
$\bar J$ be a compact invariant set. Since $\bar J$ is bounded and
$f^n(\bar J)=\bar J$, attraction gives
$\bar J\subset N_\varepsilon(\bar{\mathcal A})$ for every
$\varepsilon>0$. Compactness of $\bar{\mathcal A}$ then implies
$\bar J\subset\bar{\mathcal A}$.

	\end{Remark}
	\subsection{Function spaces and multilinear maps}
{For any integer $j\ge1$, let} $L^j(\mathbb{R}^d, \mathbb{R}^d)$ be the finite-dimensional Banach space of all $j$-multilinear maps from $(\mathbb{R}^d)^j = \mathbb{R}^d \times \cdots \times \mathbb{R}^d$ to $\mathbb{R}^d$, equipped with the standard multilinear operator norm:
	\[
	\|A\|_{L^j} := \sup_{\substack{v_1, \dots, v_j \neq 0}} \frac{|A(v_1, \dots, v_j)|}{|v_1| \cdots |v_j|} = \sup_{|v_1| \le 1, \dots, |v_j| \le 1} |A(v_1, \dots, v_j)|.
	\]
	We denote by $\mathcal{S}^j(\mathbb{R}^d, \mathbb{R}^d) \subset L^j(\mathbb{R}^d, \mathbb{R}^d)$ the closed subspace consisting of all symmetric $j$-multilinear maps, which inherits the same multilinear operator norm.
Of course, at any $x\in\mathbb R^d$, one has
$D^j\phi(x)\in\mathcal S^j(\mathbb R^d,\mathbb R^d)$.
	
	\smallskip
	
	Let $C^k(\mathbb{T}^d, \mathbb{R}^d)$ be the Banach space of real-valued, $\mathbb{Z}^d$-periodic functions of class $C^k$ on $\mathbb{R}^d$ {($k\in\mathbb Z_{\ge0}$)} equipped with the norm
\[
 \|\phi\|_{C^k}
 :=\max\left\{\|\phi\|_{C^0},\,
 \max_{1\le j\le k}\sup_{x\in\mathbb R^d}
 \|D^j\phi(x)\|_{L^j}\right\}.
\] For a continuous map $\phi \colon \mathbb{R}^d \to \mathbb{R}^d$, its $C^0$-norm is defined by $$\|\phi\|_{C^0} := \sup_{x \in \mathbb{R}^d} |\phi(x)|.$$ For $C^1$-smooth maps, the $C^1$-norm is given by $\|\phi\|_{C^1} := \max \big\{ \|\phi\|_{C^0}, \; \|D\phi\|_{C^0} \big\}$, where $\|D\phi\|_{C^0} := \sup_{x \in \mathbb{R}^d} \|D\phi(x)\|$ is determined by the matrix operator norm.

	{The Lipschitz seminorm of a map} $\phi \colon \mathbb{R}^d \to \mathbb{R}^d$ is defined by
	\[
	\|\phi\|_{\mathrm{Lip}} := \sup_{x \neq y} \frac{|\phi(x) - \phi(y)|}{|x - y|}.
	\]
	The space $C^{0,1}(\mathbb{T}^d, \mathbb{R}^d)$ denotes the Banach space of $\mathbb{Z}^d$-periodic Lipschitz continuous maps equipped with the norm $\|\phi\|_{C^{0,1}} := \|\phi\|_{C^0} + \|\phi\|_{\mathrm{Lip}}$. In the one-dimensional scalar case, $C^{0,1}(\mathbb{T}, \mathbb{R})$ forms a Banach algebra.
For $\beta\in(0,1)$, the H\"older seminorm of a vector-valued map $u$ is
\[
 [u]_\beta:=\sup_{x\ne y}\frac{|u(x)-u(y)|}{|x-y|^\beta}.
\]
For multilinear-map-valued functions, we use the same notation with the
corresponding operator norm in the numerator.

	When $r > 1$ is a non-integer, uniquely decomposed as $r = k + \beta$ with $k = \lfloor r \rfloor \ge 1$ and $\beta \in (0, 1)$, the non-integer regularity space $C^r(\mathbb{T}^d, \mathbb{R}^d)$ is identified with the H\"{o}lder space $C^{k, \beta}(\mathbb{T}^d, \mathbb{R}^d)$, endowed with the standard norm
	\[
	\|\phi\|_{C^r} = \|\phi\|_{C^{k, \beta}} := \|\phi\|_{C^k} + [D^k \phi]_{\beta}.
	\]

For $r=0$ or a real number $r\ge1$, we use the affine space
\[
 \mathrm{Id}+C^r(\mathbb T^d,\mathbb R^d)
 :=\{g=\mathrm{Id}+\varphi:\varphi\in C^r(\mathbb T^d,\mathbb R^d)\},
\]
with the metric $d_{C^r}(g_1,g_2):=\|g_1-g_2\|_{C^r}$.
This metric is well defined because $g_1-g_2$ is periodic, and the space
is complete. We use the same convention for
$\mathrm{Id}+C^{0,1}(\mathbb T^d,\mathbb R^d)$.
A torus map $\bar g$ homotopic to the identity has lifts satisfying
$g(x+k)=g(x)+k$ for $k\in\mathbb Z^d$; its $C^r$ lifts belong
to this affine space (see \cite{Her79}).
For $L\in(0,1)$ and $r\ge1$, define
\[
 L_L^{(r)}:=\{g\in\mathrm{Id}+C^r(\mathbb T^d,\mathbb R^d):
                  \|Dg-I_d\|_{C^0}\le L\}.
\]
We also use the complete metric space
\[
 L_L^{(0)}:=\{g\in\mathrm{Id}+C^{0,1}(\mathbb T^d,\mathbb R^d):
                  \|g-\mathrm{Id}\|_{\mathrm{Lip}}\le L\},
 \qquad d(g,h)=\|g-h\|_{C^0}.
\]
The superscript $0$ here refers to the metric used for the iteration;
its elements are Lipschitz lifts. Completeness follows because this
set is closed in $\mathrm{Id}+C^0(\mathbb T^d,\mathbb R^d)$.

\subsection{Twist maps, invariant graphs, and Lagrangian structures}
	
	\begin{Definition}[Dissipative Twist Map]\label{DTM}
		A dissipative twist map on the high-dimensional (infinite) cylinder $\mathbb{T}^d \times \mathbb{R}^d$ is a $C^1$ diffeomorphism $f \colon \mathbb{T}^d \times \mathbb{R}^d \to \mathbb{T}^d \times \mathbb{R}^d$ that admits a lift $F \colon \mathbb{R}^{2d} \to \mathbb{R}^{2d}$, denoted as $F(x, y) = (X(x, y), Y(x, y))$, satisfying the following conditions:
		\begin{enumerate}[label=(\roman*), font=\upshape]
			\item {(Lift condition) The map $f$ is isotopic to the identity, and $F(x+k,y)=F(x,y)+(k,0)$ for all $k\in\mathbb Z^d$;}
			\item (Twist condition) The map $(x, y) \mapsto (x, X(x, y))$ is a diffeomorphism on $\mathbb{R}^{2d}$;
			\item (Dissipative condition) There exist constants $\lambda \in (0, 1)$ and $\delta > 0$ such that for all $(x, y) \in \mathbb{R}^{2d}$:
			\[
			0 < \delta \le \det(DF(x, y)) \le \lambda.
			\]
		\end{enumerate}
	\end{Definition}
	
	Writing $\vartheta:=\sum_{i=1}^d y_i\,dx_i$ for the Liouville form,
we use the convention
\[
 \Omega:=-d\vartheta=\sum_{i=1}^d dx_i\wedge dy_i.
\]
The same notation denotes these forms on the cylinder and their lifts.

\begin{Definition}[Conformal Symplectic Twist Map]\label{def:Conformally Symplectic Twist Map}
		{A conformal symplectic twist map on the cylinder} is a $C^1$ diffeomorphism $f \colon \mathbb{T}^d \times \mathbb{R}^d \to \mathbb{T}^d \times \mathbb{R}^d$ admitting a lift $F \colon \mathbb{R}^{2d} \to \mathbb{R}^{2d}$, denoted $F(x,y) = (X(x,y), Y(x,y))$, satisfying:
		\begin{enumerate}[label=(\roman*), font=\upshape]
			\item {(Lift condition) The map $f$ is isotopic to the identity, and $F(x+k,y)=F(x,y)+(k,0)$ for all $k\in\mathbb Z^d$;}
			\item (Twist condition) The map $(x,y) \mapsto (x, X(x,y))$ is a diffeomorphism on $\mathbb{R}^{2d}$;
			\item (Conformal symplectic condition) There exists $\lambda \in (0, 1)$ such that
			\[
			{F^*\Omega=\lambda\Omega,}
			\]

where $\Omega$ is the standard symplectic form and $F^*$ denotes pullback.
 Equivalently, the 1-form
\(\Theta:=F^*(y\cdot dx)-\lambda y\cdot dx\) is closed in the distributional sense. This interpretation is needed because $F\in C^1$ only guarantees that $\Theta$ is continuous; the identity follows by approximating $F$ locally in $C^1$ by smooth maps and passing to the limit in $d\Theta=\lambda\Omega-F^*\Omega$. The map is called \emph{exact conformally symplectic} if this form is exact on \(\mathbb T^d\times\mathbb R^d\): on the lift, \(\Theta=dG\) for a function \(G(x,y)\) that is \(\mathbb Z^d\)-periodic in \(x\).  By the twist condition one can solve for \(y=y(x,X)\); the  generating function
\(S(x,X):=G(x,y(x,X))\) then satisfies
\[
 dS=Y\cdot dX-\lambda y\cdot dx,
 \qquad \lambda y=-\partial_1S(x,X),\quad Y=\partial_2S(x,X).
\]
It is diagonally periodic: \(S(x+k,X+k)=S(x,X)\) for \(k\in\mathbb Z^d\).

		\end{enumerate}
	\end{Definition}
	
	\begin{Definition}[$C^0$-Invariant Graph]\label{def:invariant_graph}
		A subset $\bar{\Gamma} \subset \mathbb{T}^d \times \mathbb{R}^d$ is called a $C^0$-invariant graph of $f$ if:
		\begin{enumerate}[label=(\roman*), font=\upshape]
			\item $\bar{\Gamma} = \{(\theta, \bar{\psi}(\theta)) \mid \theta \in \mathbb{T}^d\}$, where $\bar{\psi} = (\bar{\psi}_1, \dots, \bar{\psi}_d) \colon \mathbb{T}^d \to \mathbb{R}^d$ is a continuous function;
			\item $\bar{\Gamma}$ is invariant under $f$, that is, $f(\bar{\Gamma}) = \bar{\Gamma}$.
		\end{enumerate}
	\end{Definition}
	
	\begin{Definition}[$C^0$-Invariant Lagrangian Graph]\label{lgr}

The invariant graph $\bar\Gamma$ is called a $C^0$-invariant Lagrangian
graph if the 1-form
$\bar\psi(\theta)\cdot d\theta:=\sum_{i=1}^d\bar\psi_i(\theta)d\theta_i$
is closed in the sense of distributions. Equivalently, there exist
$c\in\mathbb R^d$ and $\bar\eta\in C^1(\mathbb T^d,\mathbb R)$ such that
$\bar\psi=c+D\bar\eta$, where
$D\bar\eta=(\partial_{\theta_1}\bar\eta,\dots,
\partial_{\theta_d}\bar\eta)^T$.

	\end{Definition}

Let $\Omega^k(\mathbb T^d)$ be the space of smooth differential
$k$-forms, and let $d$ be the exterior derivative, with $d^2=0$.
The spaces of closed and exact smooth 1-forms are, respectively,
\[
 \mathcal Z^1(\mathbb T^d)
 :=\ker(d:\Omega^1(\mathbb T^d)\to\Omega^2(\mathbb T^d)),
 \qquad
 \mathcal B^1(\mathbb T^d):=d\Omega^0(\mathbb T^d).
\]
The first de Rham cohomology group is
\begin{equation}\label{eq:de_Rham_def}
 H^1_{\mathrm{dR}}(\mathbb T^d)
 :=\frac{\mathcal Z^1(\mathbb T^d)}{\mathcal B^1(\mathbb T^d)}.
\end{equation}
For a closed form $\Theta$, its class $[\Theta]$ vanishes if and only
if $\Theta$ is exact. The classes $[d\theta_1],\dots,[d\theta_d]$
form a basis of $H^1_{\mathrm{dR}}(\mathbb T^d)$
(see \cite[Example~18.16]{Lee12}).

For the continuous closed 1-form in Definition~\ref{lgr}, the decomposition
$\bar\psi(\theta)\cdot d\theta=c\cdot d\theta+d\bar\eta$ defines its
cohomology class by
\begin{equation}
 [\bar\psi(\theta)\cdot d\theta]
 =\sum_{i=1}^d c_i[d\theta_i]\in H^1_{\mathrm{dR}}(\mathbb T^d).
\end{equation}
The coefficients $c_i$ are the integrals over the coordinate circles,
so this class is independent of the decomposition. A Lagrangian graph
is called \emph{exact} if this class vanishes, equivalently if $c=0$.

	\subsection{Rotation number and fiber contraction}
	
	{For $d=1$, the restriction of a dissipative twist map to an invariant graph induces a circle homeomorphism. We use the following convention for its rotation number.}
	
	\begin{Definition}[Rotation Number]\label{def:rotation_number}

Let $\bar g:\mathbb T\to\mathbb T$ be an orientation-preserving circle
homeomorphism, and let $g:\mathbb R\to\mathbb R$ be a degree-one lift.
We define
\begin{equation}\label{eq:rotation_number_def}
 \rho(g):=\lim_{n\to\infty}\frac{g^n(x)-x}{n}\in\mathbb R,
 \qquad \rho(\bar g):=\rho(g)\pmod1.
\end{equation}
The real limit exists and is independent of $x$ (see \cite{Her79}).
Replacing $g$ by $g+\ell$, $\ell\in\mathbb Z$, replaces $\rho(g)$ by
$\rho(g)+\ell$. Thus $\rho(\bar g)$ is independent of the lift and is
invariant under orientation-preserving topological conjugacy.

	\end{Definition}
	{For a $C^1$ invariant graph over $\mathbb T$, the slope of its tangent line defines an invariant section of the bundle $\mathbb T\times\mathbb R$. We use the following form of the $r$-fiber contraction condition \cite{HPS77} to study the regularity of this section.}
	
	\begin{Definition}[$r$-fiber contraction]\label{def:r-fiber contraction}
		Let $\bar{\mathcal{F}} \colon \mathbb{T} \times \mathbb{R} \to \mathbb{T} \times \mathbb{R}$ be a fiber-preserving map of the form $\bar{\mathcal{F}}(\theta, s) = \bigl(\bar{h}(\theta), \bar{G}(\theta, s)\bigr)$. For a given integer $r \ge 1$, $\bar{\mathcal{F}}$ is called an \emph{$r$-fiber contraction} if it satisfies the following conditions:
		\begin{enumerate}[label=(\roman*), font=\upshape]
			\item The base map $\bar{h} \colon \mathbb{T} \to \mathbb{T}$ is a $C^r$ diffeomorphism, and the fiber component $\bar{G} \colon \mathbb{T} \times \mathbb{R} \to \mathbb{R}$ is of class $C^r$;
			\item There exists a constant $K \in (0, 1)$ such that for all $\theta \in \mathbb{T}$ and $s_1, s_2 \in \mathbb{R}$,
			\[
			|\bar{G}(\theta, s_1) - \bar{G}(\theta, s_2)| \le K |s_1 - s_2|.
			\]
			{With $\kappa:=\sup_{\theta\in\mathbb T}|D\bar h^{-1}(\theta)|$, we also require}
			\[
			K \cdot \kappa^r < 1.
			\]
		\end{enumerate}
	\end{Definition}

	\vspace{1em}

\section{Persistence of invariant graphs}\label{sec:persistence}
 \rev{Theorem~\ref{thm:graph_persistence} is based on a
Herman--Mather like formula in the present context that we shall prove below together with a contraction theorem. The proof of Theorem \ref{thm:graph_persistence} concludes Section \ref{sec:fixed_point_regularity}.} \\
Let us so consider the dissipative twist map \eqref{eq:dissipative_map},
with $M$ as in the Introduction. Set
\[
 C:=\alpha_1+\Lambda_1\alpha_2-M\alpha_1.
\]
The change of fiber coordinates $w=\Lambda_1y$ conjugates the integrable
map as
\[
 (x,w)\longmapsto(x+\alpha_1+w,\,\Lambda_1\alpha_2+Mw),
\]
so the condition $\|M\|<1$ gives contraction in the fiber coordinate $w$.

\begin{Remark}\label{rem:det}
The block determinant formula gives
\[
 \det DF_{\Lambda,\alpha}^{\phi}(x,y)
 =\det\Lambda_2\,
   \det\bigl(I_d+D\phi_1(x)-\Lambda_1\Lambda_2^{-1}D\phi_2(x)\bigr).
\]
This is a continuous periodic function of $x$, independent of $y$.
By compactness of $\mathbb T^d$, \eqref{dissp} gives a positive minimum
and a maximum strictly less than one. Hence, for this family,
\eqref{dissp} is equivalent to the existence of constants $\delta>0$
and $\lambda\in(0,1)$ such that
$0<\delta\le\det DF_{\Lambda,\alpha}^{\phi}\le\lambda<1$.
Global invertibility and the isotopy condition are part of the standing
assumption that the map is a dissipative twist map.
\end{Remark}

\subsection{The generalized Herman--Mather formula}\label{sec:herman_mather}
	{The Herman--Mather formula characterizes continuous invariant graphs by a functional equation for their induced base maps. It will define the operator used in the contraction argument.}
	
	\begin{Proposition}\label{pro:Herman}
		The map $F_{\Lambda,\alpha}^{\phi}$ admits a $C^0$-invariant graph $\Gamma = \{(x, \psi(x)) \mid x \in \mathbb{R}^d\}, \psi \in C^0(\mathbb{T}^d, \mathbb{R}^d)$ if and only if there exists a homeomorphism $g \in \mathrm{Id} + C^0(\mathbb{T}^d, \mathbb{R}^d)$ such that for all $x \in \mathbb{R}^d$,
		\begin{equation}\label{eq:Herman_formula}
			g(x) + M g^{-1}(x) = (I_d + M)x + C + \phi_1(x) - M \phi_1(g^{-1}(x)) + \Lambda_1 \phi_2(g^{-1}(x)),
		\end{equation}
		{Here $M$ and $C$ are as above. The graph function $\psi$ and the induced map $g$ are related by}
		\begin{equation}\label{eq:psi_explicit_formula}
			\psi(x) = \Lambda_1^{-1}\big(g(x) - x - \alpha_1 - \phi_1(x)\big).
		\end{equation}
	\end{Proposition}
	
	\begin{proof}
		($\Rightarrow$) Assume that $F_{\Lambda,\alpha}^{\phi}$ admits a $C^0$-invariant graph $\Gamma = \{(x, \psi(x)) \mid x \in \mathbb{R}^d\}$ with $\psi \in C^0(\mathbb{T}^d, \mathbb{R}^d)$. Consider the graph {parameterization} $\iota \colon \mathbb{R}^d \to \Gamma$, $\iota(x) = (x, \psi(x))$ and the  projection $\pi \colon \Gamma \to \mathbb{R}^d$, $\pi(x, y) = x$. Since $\psi$ is continuous, $\iota$ and $\pi$ are continuous mutual inverses, hence homeomorphisms. As $F_{\Lambda,\alpha}^{\phi}$ is a global diffeomorphism and the graph $\Gamma$ is invariant, the restriction $F_{\Lambda,\alpha}^{\phi}|_{\Gamma} \colon \Gamma \to \Gamma$ is a homeomorphism onto $\Gamma$. The induced base map
		\[
		g := \pi \circ (F_{\Lambda,\alpha}^{\phi}|_{\Gamma}) \circ \iota \colon \mathbb{R}^d \to \mathbb{R}^d
		\]
		is a composition of homeomorphisms, confirming that $g$ is a global homeomorphism of $\mathbb{R}^d$ satisfying \[F_{\Lambda,\alpha}^{\phi}(x, \psi(x)) = (g(x), \psi(g(x)))\] for all $x \in \mathbb{R}^d$. In components, this reads:
		\begin{align}
			x + \alpha_1 + \Lambda_1 \psi(x) + \phi_1(x) &= g(x), \label{eq:prop_1} \\
			\alpha_2 + \Lambda_2 \psi(x) + \phi_2(x) &= \psi(g(x)). \label{eq:prop_2}
		\end{align}
		Since $\Lambda_1$ is invertible, equation \eqref{eq:prop_1} yields
		\begin{equation}\label{eq:psi_x}
			\psi(x) = \Lambda_1^{-1} \big( g(x) - x - \alpha_1 - \phi_1(x) \big).
		\end{equation}
		Because $\psi$ and $\phi_1$ are $\mathbb{Z}^d$-periodic, the displacement $g(x) - x = \alpha_1 + \Lambda_1 \psi(x) + \phi_1(x)$ is $\mathbb{Z}^d$-periodic, confirming that $g \in \mathrm{Id} + C^0(\mathbb{T}^d, \mathbb{R}^d)$.
		Since $g$ is a homeomorphism of $\mathbb{R}^d$, replacing $x$ by $g^{-1}(x)$ in \eqref{eq:prop_2} gives
		\begin{equation}\label{eq:prop_2_shifted}
			\alpha_2 + \Lambda_2 \psi(g^{-1}(x)) + \phi_2(g^{-1}(x)) = \psi(x).
		\end{equation}
		Substituting the expression \eqref{eq:psi_x} evaluated at $g^{-1}(x)$ and at $x$ into \eqref{eq:prop_2_shifted} and multiplying through by $\Lambda_1$, we obtain
		\[
		\Lambda_1 \alpha_2 + \Lambda_1 \Lambda_2 \Lambda_1^{-1} \big( x - g^{-1}(x) - \alpha_1 - \phi_1(g^{-1}(x)) \big) + \Lambda_1 \phi_2(g^{-1}(x)) = g(x) - x - \alpha_1 - \phi_1(x).
		\]
		{Rearranging the terms yields \eqref{eq:Herman_formula}.}
		
		\medskip
		($\Leftarrow$) Conversely, assume that there exists a homeomorphism $g \in \mathrm{Id} + C^0(\mathbb{T}^d, \mathbb{R}^d)$ satisfying \eqref{eq:Herman_formula}. Define $\psi:\mathbb{R}^d \to \mathbb{R}^d$ by \eqref{eq:psi_explicit_formula} and set $\Gamma = \{(x, \psi(x)) \mid x \in \mathbb{R}^d\}$. By the definition of $\psi$, the first coordinate of $F_{\Lambda,\alpha}^{\phi}(x, \psi(x))$ is identically $g(x)$. Substituting $x = g(u)$ into \eqref{eq:Herman_formula} and multiplying by $\Lambda_1^{-1}$ recovers $\alpha_2 + \Lambda_2 \psi(u) + \phi_2(u) = \psi(g(u))$ for all $u \in \mathbb{R}^d$. Hence $F_{\Lambda,\alpha}^{\phi}(x, \psi(x)) = (g(x), \psi(g(x)))$. Since $g$ is surjective, $F_{\Lambda,\alpha}^{\phi}(\Gamma) = \Gamma$, confirming that $\Gamma$ is an invariant graph.
	\end{proof}

For a homeomorphism $g\in\mathrm{Id}+C^0(\mathbb T^d,\mathbb R^d)$, define
\begin{equation}\label{eq:HM_operator}
 T(g)(x):=(I_d+M)x+C+\phi_1(x)-M\phi_1(g^{-1}(x))
            +\Lambda_1\phi_2(g^{-1}(x))-Mg^{-1}(x).
\end{equation}
Proposition~\ref{pro:Herman} identifies the continuous invariant graphs
with the fixed points of $T$ among these lifts.

\subsection{The contraction theorem}\label{sec:contraction_statement}
For $r\ge1$ and an admissible $L$, let $B_\phi^{(r)}$ denote the pair
$(\phi_1,\phi_2)\in(C^r(\mathbb T^d,\mathbb R^d))^2$ satisfying
\eqref{eq:persistence_smallness}. The spaces $L_L^{(r)}$, including
$L_L^{(0)}$, were defined in Section~\ref{sec:notation}.

\begin{Theorem}[Contraction theorem]\label{thm:DSM_Cr_smoothness}
Let $r\ge1$, and suppose that $F_{\Lambda,\alpha}^{\phi}$ and $L$
satisfy the hypotheses of Theorem~\ref{thm:graph_persistence}.
Then $T$ preserves $L_L^{(0)}$ and is a strict contraction in the
$C^0$ distance. More precisely, for $g,h\in L_L^{(0)}$,
\[
 \|T(g)-T(h)\|_{C^0}
 \le\frac{\|M\|+\|\Lambda_1\phi_2-M\phi_1\|_{\mathrm{Lip}}}{1-L}
       \|g-h\|_{C^0},
 \qquad
 \frac{\|M\|+\|\Lambda_1\phi_2-M\phi_1\|_{\mathrm{Lip}}}{1-L}<1.
\]
Its unique fixed point $g$ belongs to $L_L^{(r)}$.
The operator also preserves $L_L^{(r)}$, and the iterates
$g_0=\mathrm{Id}$, $g_{n+1}=T(g_n)$ converge uniformly to $g$.
\end{Theorem}

The proof first constructs the fixed point in $L_L^{(0)}$, then proves
its regularity through the derivative equations. We give these two
steps separately.

\subsection{Persistence of the Lipschitz invariant graph}
	{We first construct a Lipschitz invariant graph, then prove its higher regularity.}
	
	\begin{Lemma}\label{lem:Lipschitz-I-G}
		Let $0 < L < 1 - \|M\|$. Suppose that $F_{\Lambda,\alpha}^\phi$ is a dissipative twist diffeomorphism lift satisfying $0 < \delta \le \det DF_{\Lambda,\alpha}^\phi(x, y) \le \lambda < 1$. If the perturbation $(\phi_1, \phi_2) \in B_\phi^{(0)}$, where
		\[
		B_\phi^{(0)} = \left\{ (\phi_1, \phi_2) \in (C^{0,1}(\mathbb{T}^d, \mathbb{R}^d))^2 \;\left|\; \|\phi_1\|_{\mathrm{Lip}} + \frac{1}{1-L} \|\Lambda_1 \phi_2 - M \phi_1\|_{\mathrm{Lip}} \le \frac{L(1 - L - \|M\|)}{1-L} \right. \right\},
		\]
		then there exists a unique $g\in L_L^{(0)}$ such that $T(g)=g$.
        Moreover, the map $F_{\Lambda,\alpha}^{\phi}$ admits a unique Lipschitz invariant graph $\Gamma = \{(x, \psi(x)) \mid x \in \mathbb{R}^d\}$, whose projection $\bar{\Gamma}$ on $\mathbb{T}^d \times \mathbb{R}^d$ is a global attractor.
	\end{Lemma}
	
	\begin{proof}
		For any $g \in L_L^{(0)}$, the condition $\|g - {\mathrm{Id}}\|_{\mathrm{Lip}} \le L < 1$ implies that $g$ is a bi-Lipschitz homeomorphism satisfying $\|g^{-1}\|_{\mathrm{Lip}} \le (1-L)^{-1}$ and $\|{\mathrm{Id}} - g^{-1}\|_{\mathrm{Lip}} \le L(1-L)^{-1}$. Since $L_L^{(0)}$ is a closed subset of the complete metric space $(\mathrm{Id} + C^0(\mathbb{T}^d, \mathbb{R}^d), d_{C^0})$, the space $(L_L^{(0)}, d_{C^0})$ is complete.
		
		\medskip \noindent\textit{Step 1: $T(L_L^{(0)}) \subset L_L^{(0)}$.}\par
		Take any $g \in L_L^{(0)}$. Rewriting $T(g)(x) - x$, we have:
		\[
		T(g)(x) - x = M(x - g^{-1}(x)) + C + \phi_1(x) + (\Lambda_1 \phi_2 - M \phi_1)(g^{-1}(x)).
		\]
		Since $\phi_1, \phi_2$ are $\mathbb{Z}^d$-periodic and ${g\in L_L^{(0)}}$, for any $k \in \mathbb{Z}^d$ we have $g^{-1}(x+k) = g^{-1}(x) + k$, which yields $T(g)(x+k) - (x+k) = T(g)(x) - x$. Thus $T(g) - \mathrm{Id}$ is $\mathbb{Z}^d$-periodic. Furthermore, for any $x, y \in \mathbb{R}^d$,
		\begin{align*}
			|(T(g)(x) - x) - (T(g)(y) - y)|
			&\le \|M\| \cdot |(x - g^{-1}(x)) - (y - g^{-1}(y))| + |\phi_1(x) - \phi_1(y)| \\
			&\quad + |(\Lambda_1 \phi_2 - M \phi_1)(g^{-1}(x)) - (\Lambda_1 \phi_2 - M \phi_1)(g^{-1}(y))| \\
			&\le \left( \|M\| \frac{L}{1-L} + \|\phi_1\|_{\mathrm{Lip}} + \frac{1}{1-L} \|\Lambda_1 \phi_2 - M \phi_1\|_{\mathrm{Lip}} \right) |x - y|.
		\end{align*}
		By the definition of $B_\phi^{(0)}$, the parenthesized factor is bounded by
		\[
		\|M\| \frac{L}{1-L} + \frac{L(1 - L - \|M\|)}{1-L} = \frac{L(1-L)}{1-L} = L.
		\]
		Hence $\|T(g) - {\mathrm{Id}}\|_{\mathrm{Lip}} \le L$, confirming that $T(L_L^{(0)}) \subset L_L^{(0)}$.
		
		\medskip \noindent\textit{Step 2: Strict contractivity of $T$.}\par
		For any $f, g \in L_L^{(0)}$,
		\[
		T(f)(x) - T(g)(x) = -M(f^{-1}(x) - g^{-1}(x)) + \Big[ (\Lambda_1 \phi_2 - M \phi_1)(f^{-1}(x)) - (\Lambda_1 \phi_2 - M \phi_1)(g^{-1}(x)) \Big].
		\]
		Taking the Euclidean norm yields
		\[
		|T(f)(x) - T(g)(x)| \le \Big( \|M\| + \|\Lambda_1 \phi_2 - M \phi_1\|_{\mathrm{Lip}} \Big) |f^{-1}(x) - g^{-1}(x)|.
		\]
		Setting $u = f^{-1}(x)$, we have $x = f(u) = g(g^{-1}(x))$, and
		\[
		|f^{-1}(x) - g^{-1}(x)| = |g^{-1}(g(u)) - g^{-1}(f(u))| \le \|g^{-1}\|_{\mathrm{Lip}} |g(u) - f(u)| \le \frac{1}{1-L} \|f - g\|_{C^0}.
		\]
		Taking the supremum over $x \in \mathbb{R}^d$ gives
		\[
		d_{C^0}(T(f), T(g)) \le \left( \frac{\|M\| + \|\Lambda_1 \phi_2 - M \phi_1\|_{\mathrm{Lip}}}{1-L} \right) d_{C^0}(f, g).
		\]
		Using the condition defining $B_\phi^{(0)}$,
		\[
		\frac{\|M\| + \|\Lambda_1 \phi_2 - M \phi_1\|_{\mathrm{Lip}}}{1-L} \le \frac{\|M\| + L(1 - L - \|M\|)}{1-L} = \|M\| + L < 1.
		\]
		Thus $T$ is a strict contraction on $(L_L^{(0)}, d_{C^0})$. By the Banach fixed-point theorem, there exists a unique fixed point $g \in L_L^{(0)}$.
		
		\medskip \noindent\textit{Step 3: Global attraction and uniqueness of the invariant graph.}\par
		Let $\Gamma = \{(x, \psi(x)) \mid x \in \mathbb{R}^d\}$ be the invariant graph associated with $g$, with $\psi$ defined by \eqref{eq:psi_explicit_formula}, and let $\bar{\Gamma} = \{(\theta, \bar{\psi}(\theta)) \mid \theta \in \mathbb{T}^d\}$ be its projection on $\mathbb{T}^d \times \mathbb{R}^d$.
		
		Let $\bar{B} \subset \mathbb{T}^d \times \mathbb{R}^d$ be any bounded set. Then there exists $R > 0$ such that all initial points $(x^0, y^0) \in \mathbb{R}^{2d}$ projecting into $\bar{B}$ satisfy $|y^0| \le R$. The initial vertical error $e^0 := y^0 - \psi(x^0)$ satisfies $|e^0| \le R + \|\psi\|_{C^0} =: C_B < \infty$.
		
		Let $(x^n, y^n) = (F_{\Lambda,\alpha}^\phi)^n(x^0, y^0)$ and $e^n := y^n - \psi(x^n)$. Using the invariance relation \[\psi(g(x^{n-1})) = \alpha_2 + \Lambda_2 \psi(x^{n-1}) + \phi_2(x^{n-1}),\quad y^n = \alpha_2 + \Lambda_2 y^{n-1} + \phi_2(x^{n-1}),\] we obtain
		\[
		e^n = \Lambda_2 e^{n-1} - \Big[ \psi(x^n) - \psi(g(x^{n-1})) \Big].
		\]
		Introducing the transformed error variable $\tilde{e}^n := \Lambda_1 e^n$ and multiplying by $\Lambda_1$ gives
		\begin{equation}\label{eq:tilde_e_recurrence_clean}
			\tilde{e}^n = M \tilde{e}^{n-1} - \Lambda_1 \Big[ \psi(x^n) - \psi(g(x^{n-1})) \Big].
		\end{equation}
		Notice that
		\begin{align*}
		x^n - g(x^{n-1}) &= \big(x^{n-1} + \alpha_1 + \Lambda_1 y^{n-1} + \phi_1(x^{n-1})\big) - \big(x^{n-1} + \alpha_1 + \Lambda_1 \psi(x^{n-1}) + \phi_1(x^{n-1})\big)\\
& = \Lambda_1 e^{n-1} = \tilde{e}^{n-1}.
		\end{align*}
		Furthermore, because $g$ is the fixed point of $T$, Proposition~\ref{pro:Herman} implies
		\[
		\Lambda_1 \psi(x) = M\big(x - {g^{-1}(x)}\big) + C - \alpha_1 + (\Lambda_1 \phi_2 - M \phi_1)({g^{-1}(x)}).
		\]
		Using $\|{\mathrm{Id} - g^{-1}}\|_{\mathrm{Lip}} \le \frac{L}{1-L}$ and $\|{g^{-1}}\|_{\mathrm{Lip}} \le \frac{1}{1-L}$, the Lipschitz constant of $\Lambda_1 \psi$ is bounded by
		\begin{equation}\label{eq:Lambda1_psi_lip_bound}
			\|\Lambda_1 \psi\|_{\mathrm{Lip}} \le \|M\| \frac{L}{1-L} + \frac{\|\Lambda_1 \phi_2 - M \phi_1\|_{\mathrm{Lip}}}{1-L} = \frac{\|M\|L + \|\Lambda_1 \phi_2 - M \phi_1\|_{\mathrm{Lip}}}{1-L}.
		\end{equation}
		Substituting \eqref{eq:Lambda1_psi_lip_bound} into \eqref{eq:tilde_e_recurrence_clean} yields
		\begin{align*}
			|\tilde{e}^n| &\le \|M\| |\tilde{e}^{n-1}| + \|\Lambda_1 \psi\|_{\mathrm{Lip}} |x^n - g(x^{n-1})| \\
			&= \big( \|M\| + \|\Lambda_1 \psi\|_{\mathrm{Lip}} \big) |\tilde{e}^{n-1}| \\
			&\le \frac{\|M\| + \|\Lambda_1 \phi_2 - M \phi_1\|_{\mathrm{Lip}}}{1-L} |\tilde{e}^{n-1}| =: \tilde{\kappa} |\tilde{e}^{n-1}|.
		\end{align*}
		By the condition defining $B_\phi^{(0)}$, $\tilde{\kappa} \le \|M\| + L < 1$. Iterating this bound gives
		\[
		|e^n| \le \|\Lambda_1^{-1}\| |\tilde{e}^n| \le \|\Lambda_1^{-1}\| \|\Lambda_1\| C_B \tilde{\kappa}^n \xrightarrow{n \to \infty} 0.
		\]
		On the cylinder $\mathbb{T}^d \times \mathbb{R}^d$, the distance satisfies $\bar{d}\big((\theta^n, y^n), \bar{\Gamma}\big) \le |y^n - \psi(x^n)| = |e^n| \to 0$ uniformly on $\bar{B}$. {By} Definition~\ref{def:Global Attractor}, $\bar{\Gamma}$ is a global attractor. By Remark~\ref{rem:attractor_maximality}, $\bar{\Gamma}$ contains every compact invariant set on the cylinder. In particular, if $\bar{\Gamma}'$ is any continuous invariant graph, its compactness and invariance imply $\bar{\Gamma}' \subset \bar{\Gamma}$; since both $\bar{\Gamma}'$ and $\bar{\Gamma}$ are graphs over the entire base torus $\mathbb{T}^d$, we necessarily have $\bar{\Gamma}' = \bar{\Gamma}$, establishing the uniqueness of the continuous invariant graph.
	\end{proof}
		
\begin{Remark}
The attraction and global uniqueness above use
\[
\widetilde\kappa=
\frac{\|M\|+\|\Lambda_1\phi_2-M\phi_1\|_{\mathrm{Lip}}}{1-L}
\le\|M\|+L<1.
\]
The hypotheses of Theorem~\ref{thm:DSM_Cr_smoothness}, including its \(C^1\) endpoint, imply \(L<1-\|M\|\). Thus this argument also gives uniqueness among all continuous invariant graphs in that theorem.
\end{Remark}

	{The following convergence lemmas will be applied to the derivative iterates.}
	
	\begin{Lemma}\label{lem:convergence_of_derivatives}
	Let {$\mathcal{X}$ be a nonempty closed subset of a Banach space, with the induced norm distance}, $(\mathcal{Y}, d)$ be a metric space, and $g \in \mathcal{Y}$. Suppose the operator $\mathcal{F}: \mathcal{X} \times \mathcal{Y} \to \mathcal{X}$ satisfies:
	\begin{itemize}
		\item[(i)] There exists $\rho \in (0, 1)$ such that for all $u, w \in \mathcal{X}$ and {$h \in \mathcal{Y}$},
		\[
		{\|\mathcal{F}(u, h) - \mathcal{F}(w, h)\|_{\mathcal{X}} \le \rho \|u - w\|_{\mathcal{X}}.}
		\]
		\item[(ii)] For every fixed $v \in \mathcal{X}$, the mapping {$h\mapsto\mathcal F(v,h)$} is continuous at $g$.
	\end{itemize}
	Then $\mathcal{F}(\cdot, g)$ admits a unique fixed point $v \in \mathcal{X}$. Furthermore, if a sequence $\{g_n\} \subset \mathcal{Y}$ converges to $g$, then for any initial state $v_0 \in \mathcal{X}$, the sequence $v_{n+1} = \mathcal{F}(v_n, g_n)$ converges to $v$ in $\mathcal{X}$.
\end{Lemma}
	
	\begin{proof}
	The existence and uniqueness of $v$ follow from the Banach fixed-point theorem applied to $\mathcal{F}(\cdot, g)$. By the triangle inequality and condition (i),
	\[
	\|v_{n+1} - v\|_{\mathcal{X}} \le \|\mathcal{F}(v_n, g_n) - \mathcal{F}(v, g_n)\|_{\mathcal{X}} + \|\mathcal{F}(v, g_n) - \mathcal{F}(v, g)\|_{\mathcal{X}} \le \rho \|v_n - v\|_{\mathcal{X}} + \delta_n,
	\]
	where $\delta_n := \|\mathcal{F}(v, g_n) - \mathcal{F}(v, g)\|_{\mathcal{X}} \to 0$ as $n \to \infty$ by condition (ii). Iterating this recurrence yields $\|v_n - v\|_{\mathcal{X}} \to 0$ as $n \to \infty$.
	\end{proof}
	
	{
\begin{Lemma}[Strict inverse contraction]\label{lem:strict_inverse_endpoint}
Let $0<L<1$, set $s=1-L$, and equip real matrices with the Euclidean
operator norm.  On the compact set
\[
\mathcal K_L:=\{A\in\operatorname{Mat}_d(\mathbb R):\|A-I_d\|\le L\},
\]
every matrix is invertible, and for distinct $A,B\in\mathcal K_L$,
\begin{equation}\label{eq:strict_inverse_endpoint}
s^2\|A^{-1}-B^{-1}\|<\|A-B\|.
\end{equation}
Moreover, there is a nondecreasing, upper semicontinuous function
$\omega:[0,2L]\to[0,2L]$ such that $\omega(0)=0$,
$\omega(t)<t$ for $t>0$, and
\begin{equation}\label{eq:inverse_uniform_modulus}
s^2\|A^{-1}-B^{-1}\|\le\omega(\|A-B\|)
\quad(A,B\in\mathcal K_L).
\end{equation}
\end{Lemma}

\begin{proof}
For $A\in\mathcal K_L$ and $v\in\mathbb R^d$, the triangle inequality gives
\[
|Av|=|v+(A-I_d)v|\ge |v|-|(A-I_d)v|
\ge(1-L)|v|=s|v|.
\]
Since $s>0$, the kernel of $A$ is trivial, so $A$ is invertible in
finite dimension. Applying this estimate to $v=A^{-1}u$ gives
$\|A^{-1}\|\le s^{-1}$. For any $A,B\in\mathcal K_L$, we recall the inverse identity
\[
A^{-1}-B^{-1}=A^{-1}(B-A)B^{-1}.
\]
Sub-multiplicativity of the operator norm yields
\[
\|A^{-1}-B^{-1}\|
\le\|A^{-1}\|\,\|A-B\|\,\|B^{-1}\|
\le s^{-2}\|A-B\|,
\]
which is the non-strict version of \eqref{eq:strict_inverse_endpoint}.
It also proves that matrix inversion is continuous on $\mathcal K_L$.

We now exclude equality for distinct matrices.
Suppose equality holds for $A\ne B$, and put $\delta=\|A-B\|>0$.
The operator norm is attained on the compact unit sphere, so there
is a unit vector $\xi$ with
\[
|A^{-1}(A-B)B^{-1}\xi|=s^{-2}\delta.
\]
Set $y=B^{-1}\xi$ and $z=(A-B)y$. Equality must hold at every step in
\[
s^{-2}\delta=|A^{-1}z|
\le s^{-1}|z|\le s^{-1}\delta|y|
\le s^{-2}\delta.
\]
In particular, $z\ne0$, $|y|=s^{-1}$, $|By|=1=s|y|$, and
$|A(A^{-1}z)|=s|A^{-1}z|$.
For any $B\in\mathcal K_L$ and nonzero $w$, equality in
\[
|Bw|=|w+(B-I_d)w|\ge |w|-|(B-I_d)w|\ge s|w|
\]
requires $|(B-I_d)w|=L|w|$ and equality in the reverse triangle
inequality. Since $L<1$, this means $(B-I_d)w=-Lw$, hence $Bw=sw$.
Applying this observation first to $B,y$ gives $By=sy$, and then
to $A,A^{-1}z$ gives $z=sA^{-1}z$. Multiplying the latter identity
by $A$ gives $Az=sz$.

The norm bound on $A-I_d$ now also forces $A^Tz=sz$.
Indeed, let $R=(I_d-A)/L$. Then $\|R\|\le1$ and $Rz=z$.
Using $\|R^T\|=\|R\|$ for the Euclidean operator norm, we compute
\[
|R^Tz-z|^2
=|R^Tz|^2+|z|^2-2\langle R^Tz,z\rangle
=|R^Tz|^2-|z|^2\le0.
\]
Thus $R^Tz=z$, and $A^Tz=sz$. Consequently,
\[
|z|^2=\langle(A-B)y,z\rangle
=\langle(A-sI_d)y,z\rangle
=\langle y,(A^T-sI_d)z\rangle=0,
\]
contradicting $z\ne0$. This proves \eqref{eq:strict_inverse_endpoint}.

For $0\le t\le2L$, define
\[
\mathcal P_t:=\{(A,B)\in\mathcal K_L\times\mathcal K_L:
\|A-B\|\le t\},\qquad
\omega(t):=\max_{(A,B)\in\mathcal P_t}
s^2\|A^{-1}-B^{-1}\|.
\]
The set $\mathcal K_L$ is compact because it is bounded and closed
in a finite-dimensional space. The set $\mathcal P_t$ is also
compact, since its additional constraint is preserved under limits,
and it is nonempty because it contains $(I_d,I_d)$.
The objective is continuous by the inverse estimate above.
Hence the maximum is attained at a pair $(A_t,B_t)$.
For $t>0$, if $A_t=B_t$, then $\omega(t)=0<t$; otherwise,
\[
\omega(t)=s^2\|A_t^{-1}-B_t^{-1}\|
<\|A_t-B_t\|\le t.
\]
Attainment of the maximum is what preserves strictness here.
Also, $\mathcal P_{t_1}\subset\mathcal P_{t_2}$ for $t_1\le t_2$,
so $\omega$ is nondecreasing. At $t=0$ only pairs with $A=B$ are
admissible, giving $\omega(0)=0$. In particular,
$0\le\omega(t)\le t\le2L$.

To prove upper semicontinuity, take any $t_j\to t$ in $[0,2L]$.
Since $\omega$ is bounded, choose a subsequence along which
$\omega(t_j)$ converges to $\limsup_j\omega(t_j)$, and choose a
maximizing pair $(A_j,B_j)\in\mathcal P_{t_j}$ for each of these
parameters. A further subsequence converges to a pair
$(A_*,B_*)\in\mathcal K_L\times\mathcal K_L$ by compactness.
Passing to the limit in $\|A_j-B_j\|\le t_j$ gives
$\|A_*-B_*\|\le t$, so this pair is admissible for $\mathcal P_t$.
Continuity of inversion then gives
\[
\limsup_{j\to\infty}\omega(t_j)
=s^2\|A_*^{-1}-B_*^{-1}\|\le\omega(t).
\]
This is precisely upper semicontinuity, including at $t=0$ and $t=2L$.
Finally, every pair $A,B\in\mathcal K_L$ belongs to
$\mathcal P_{\|A-B\|}$; taking its value in the defining maximum
proves \eqref{eq:inverse_uniform_modulus}.
\end{proof}

\begin{Lemma}[Convergence with a strict contraction modulus]
\label{lem:nonlinear_convergence_endpoint}
Let $(\mathcal X,d_{\mathcal X})$ be a nonempty complete metric space
of finite diameter at most $D$, and let $\omega:[0,D]\to[0,D]$
be nondecreasing and upper semicontinuous, with
$\omega(0)=0$ and $\omega(t)<t$ for $t>0$.
Let $(\mathcal Y,d_{\mathcal Y})$ be a metric space.
Suppose $\mathcal F:\mathcal X\times\mathcal Y\to\mathcal X$ satisfies
\[
d_{\mathcal X}(\mathcal F(u,h),\mathcal F(w,h))
\le\omega(d_{\mathcal X}(u,w))
\]
for all $u,w\in\mathcal X$ and $h\in\mathcal Y$.
Fix $g\in\mathcal Y$, and suppose
$h\mapsto\mathcal F(v,h)$ is continuous at $g$ for each fixed
$v\in\mathcal X$.
Then $\mathcal F(\cdot,g)$ has a unique fixed point $v$.

For every sequence $(g_n)_{n\ge0}\subset\mathcal Y$ with $g_n\to g$
and every initial value $v_0\in\mathcal X$, the sequence defined by
\[
v_{n+1}:=\mathcal F(v_n,g_n),\qquad n\ge0,
\]
converges to $v$. The initial value $v_0$ is arbitrary; all subsequent
terms $v_n$ are determined by this recurrence.

\end{Lemma}

\begin{proof}
If $D=0$, the space $\mathcal X$ consists of one point, and the
assertion is immediate. Otherwise set $d_n=\omega^{\circ n}(D)$,
where $\omega^{\circ n}$ denotes the $n$-fold composition of $\omega$
with itself and $\omega^{\circ0}$ is the identity. Equivalently,
\[
d_0=D,\qquad d_{n+1}=\omega(d_n),\qquad
d_1=\omega(D),\quad d_2=\omega(\omega(D)).
\]
Since $0\le\omega(t)\le t$, the sequence $d_n$ is nonincreasing
and has a limit $d\ge0$. If $d>0$, upper semicontinuity gives
\[
d=\lim_{n\to\infty}d_{n+1}
=\limsup_{n\to\infty}\omega(d_n)\le\omega(d)<d,
\]
a contradiction. Thus $d_n\to0$.

Write $Q=\mathcal F(\cdot,g)$. Repeated application of the
contraction estimate and monotonicity of $\omega$ gives, for all
$u,w\in\mathcal X$ and $n\ge0$,
\[
d_{\mathcal X}(Q^nu,Q^nw)
\le\omega^{\circ n}(d_{\mathcal X}(u,w)).
\]
Here $Q^n$ likewise denotes $n$ iterations of $Q$.
In particular, since the diameter is at most $D$, for any
$w\in\mathcal X$ and $m,n\ge0$,
\[
d_{\mathcal X}(Q^{n+m}w,Q^nw)
\le\omega^{\circ n}(d_{\mathcal X}(Q^mw,w))
\le\omega^{\circ n}(D)=d_n.
\]
The bound tends to zero independently of $m$, so $Q^nw$ is
Cauchy. Completeness gives a limit $v\in\mathcal X$.
Moreover, $d_{\mathcal X}(Qu,Qw)\le d_{\mathcal X}(u,w)$, so $Q$
is continuous. Passing to the limit in $Q(Q^nw)=Q^{n+1}w$ gives
$Qv=v$. If $\widetilde v$ were another fixed point at a positive
distance from $v$, then
\[
d_{\mathcal X}(v,\widetilde v)
\le\omega(d_{\mathcal X}(v,\widetilde v))
<d_{\mathcal X}(v,\widetilde v),
\]
which is impossible. Thus the fixed point is unique.

{Fix a sequence $g_n\to g$ and an arbitrary $v_0\in\mathcal X$,
and define $v_{n+1}=\mathcal F(v_n,g_n)$. Since $\mathcal F$ takes
values in $\mathcal X$, every iterate belongs to $\mathcal X$. Set}
\[
e_n=d_{\mathcal X}(v_n,v),\qquad
\varepsilon_n=d_{\mathcal X}(\mathcal F(v,g_n),\mathcal F(v,g)).
\]
The assumed parameter continuity gives $\varepsilon_n\to0$.
{Using the fixed-point identity $\mathcal F(v,g)=v$,
the defining recurrence for $v_n$, and the triangle inequality, we obtain}
\[
e_{n+1}\le
d_{\mathcal X}(\mathcal F(v_n,g_n),\mathcal F(v,g_n))
+\varepsilon_n
\le\omega(e_n)+\varepsilon_n.
\]
Set $E=\limsup_n e_n\in[0,D]$, and choose $n_j\to\infty$ with
$e_{n_j+1}\to E$. Since $0\le e_n\le D$, a further subsequence
satisfies $e_{n_j}\to E'$ for some $E'\le E$.
Upper semicontinuity followed by monotonicity yields
\[
E\le\limsup_{j\to\infty}\omega(e_{n_j})
\le\omega(E')\le\omega(E).
\]
If $E>0$, this contradicts $\omega(E)<E$. Hence $E=0$, and the
nonnegative sequence $e_n$ converges to zero, as required.
\end{proof}
}

\subsection{Regularity of the fixed point}\label{sec:fixed_point_regularity}

Lemma~\ref{lem:Lipschitz-I-G} gives a Lipschitz invariant graph.
We prove its $C^r$ regularity by showing uniform convergence of the
derivatives of the Herman--Mather iterates.

The $C^r$ section theorem in \cite[Theorem~3.5]{HPS77} proves
regularity of invariant sections for $r$-fiber contractions;
Definition~\ref{def:r-fiber contraction} records the trivial-bundle
version used in Section~\ref{sec:regularity}. In the bundle formulation, a $C^r$ map
$f:E\to E$ on a complete bundle over a compact $C^r$ manifold $X$
covers a $C^r$ diffeomorphism $h:X\to X$. The contraction conditions
are
\begin{equation}\label{kalpha}
\sup_x k_x<1,\qquad \sup_x k_x\kappa_x^r<1,
\end{equation}
where $k_x$ is the contraction constant on the fiber $E_x$ and
$\kappa_x=\|D_{h(x)}h^{-1}\|$. This section theorem also underlies
the persistence of compact $r$-normally hyperbolic invariant manifolds
in \cite[Theorem~4.1]{HPS77}. The local parameter version used in
Proposition~\ref{prop:irrational_Lip_not_C1} is stated and proved in
Lemma~\ref{lem:local_parameter_persistence}.

For the present model, differentiation of $g_{n+1}=T(g_n)$ gives
$Dg_{n+1}=F_1(Dg_n,g_n)$. At the $C^1$ endpoint,
Lemma~\ref{lem:strict_inverse_endpoint} gives a nonlinear contraction
modulus $\omega(t)<t$, and
Lemma~\ref{lem:nonlinear_convergence_endpoint} yields uniform
convergence of the derivatives. Higher integer derivatives are
handled inductively using Fa\`a di Bruno's formula and
Lemma~\ref{lem:convergence_of_derivatives}. The contraction factors
depend on the $C^1$ norms of the perturbations; higher derivatives
enter the remainders. For $r=k+\beta$, a uniform bound on
$[D^k g_n]_\beta$ gives the required H\"older regularity of the limit.

	{
We now complete the proof of Theorem~\ref{thm:DSM_Cr_smoothness}.
Throughout the proof, $L_L^{(r)}$ and $B_\phi^{(r)}$ are as defined
above.  {For this proof, abbreviate}
\[
q=\|M\|,\qquad s=1-L,\qquad
a=\|\Lambda_1\phi_2-M\phi_1\|_{C^1},\qquad
b=\|\phi_1\|_{C^1}.
\]
The first hypothesis is $b+a/s\le L(s-q)/s$.
For $r=1$ the second hypothesis is $q+a\le s^2$;
for $r>1$ it is $q+a<s^{r+1}$.
In every case $q+a\le s^2$ and
$L\le1-\sqrt q<1-q$, so
Lemma~\ref{lem:Lipschitz-I-G} applies.
}

		\medskip \noindent\textit{Step 1: $T(L_L^{(r)}) \subset L_L^{(r)}$.}\par
		For any $g \in L_L^{(r)}$, $\|Dg - I_d\|_{C^0} \le L < 1$. Thus $g$ is a $C^r$ diffeomorphism, and its inverse satisfies
		\[
		\|Dg^{-1}\|_{C^0} \le {s^{-1}}, \qquad \|I_d - Dg^{-1}\|_{C^0} \le {L/s}.
		\]
		Differentiating $T(g)(x) - x$ gives
		\[
		DT(g)(x) - I_d = M(I_d - Dg^{-1}(x)) + D\phi_1(x) + [\Lambda_1 D\phi_2(g^{-1}(x)) - M D\phi_1(g^{-1}(x))] Dg^{-1}(x).
		\]
		Taking the $C^0$-norm and using the first condition in $B_\phi^{(r)}$,
		\[
		\|DT(g) - I_d\|_{C^0} \le {qL/s+b+a/s} \le L.
		\]
		Hence $T(L_L^{(r)}) \subset L_L^{(r)}$ for all real $r \ge 1$.
		
		\medskip \noindent\textit{Step 2: Integer regularity $r = k \in \mathbb{Z}^+$.}\par
		We proceed by induction on $k$. {At each step we also prove that the displacements $g_n-\mathrm{Id}$ converge to $g-\mathrm{Id}$ in $C^k$, so that the convergence needed at the next step is part of the induction.}
		
		{
\smallskip\noindent\textbf{Base case $k=1$, including the endpoint.}
The iterates $g_{n+1}=T(g_n)$, $g_0=\mathrm{Id}$, converge in $C^0$
to the unique Lipschitz fixed point $g$.
Step~1 shows that $g_n\in L_L^{(1)}$ for every $n$.
Set $v_n=Dg_n$ and
\[
\mathcal X_L:=\{v\in C^0(\mathbb T^d,\operatorname{Mat}_d(\mathbb R)):
\|v-I_d\|_{C^0}\le L\}.
\]
This is a nonempty closed subset of a Banach space, hence a complete
metric space for the uniform norm, with diameter at most $2L$.
For $v\in\mathcal X_L$ and $h\in L_L^{(0)}$, define
\begin{align*}
F_1(v,h)(x):={}&I_d+M+D\phi_1(x)\\
&+\big[\Lambda_1D\phi_2(h^{-1}(x))
       -MD\phi_1(h^{-1}(x))-M\big]
       [v(h^{-1}(x))]^{-1}.
\end{align*}
{Differentiation gives $v_{n+1}=F_1(v_n,g_n)$,
with $v_0=Dg_0=I_d$. Thus these derivative fields form precisely
an iteration of the type considered in
Lemma~\ref{lem:nonlinear_convergence_endpoint}.}
For every $v\in\mathcal X_L$,
\[
\|F_1(v,h)-I_d\|_{C^0}
\le q\frac Ls+b+\frac as\le L,
\]
so $F_1$ maps $\mathcal X_L$ into itself.
The coefficient multiplying the inverse has norm at most
$q+a\le s^2$. For each $x$, put $\xi=h^{-1}(x)$ and apply
Lemma~\ref{lem:strict_inverse_endpoint} to $v(\xi),w(\xi)\in\mathcal K_L$:
\[
\begin{aligned}
\|F_1(v,h)(x)-F_1(w,h)(x)\|
&\le(q+a)\|v(\xi)^{-1}-w(\xi)^{-1}\|\\
&\le\omega(\|v(\xi)-w(\xi)\|)
\le\omega(\|v-w\|_{C^0}).
\end{aligned}
\]
The last inequality uses monotonicity of $\omega$.
Taking the supremum over $x$ gives
\[
\|F_1(v,h)-F_1(w,h)\|_{C^0}
\le\omega(\|v-w\|_{C^0})
\quad(v,w\in\mathcal X_L),
\]
where $\omega$ is independent of $h$.

For each fixed $v\in\mathcal X_L$, the map
$h\mapsto F_1(v,h)$ is continuous for the $C^0$ topology on
$L_L^{(0)}$.  Indeed,
\[
\|h_1^{-1}-h_2^{-1}\|_{C^0}
\le s^{-1}\|h_1-h_2\|_{C^0},
\]
and the continuous periodic coefficients and $v$ are uniformly
continuous; matrix inversion is uniformly continuous on
$\mathcal K_L$.
Lemma~\ref{lem:nonlinear_convergence_endpoint} now implies
$v_n\to v$ uniformly for a continuous matrix field $v$.
For any $x,z\in\mathbb R^d$ we can pass to the limit in
\[
g_n(x+z)-g_n(x)=\int_0^1Dg_n(x+tz)z\,dt
\]
and obtain the same identity with $g$ and $v$.
Subtracting $v(x)z$ from the limiting identity gives
\[
g(x+z)-g(x)-v(x)z
=\int_0^1\bigl(v(x+tz)-v(x)\bigr)z\,dt.
\]
Its norm is at most
$|z|\sup_{0\le t\le1}\|v(x+tz)-v(x)\|=o(|z|)$ as $z\to0$,
by continuity of $v$. Thus $g$ is differentiable with $Dg=v$,
and the continuity of this derivative gives $g\in C^1$.
Consequently $g\in L_L^{(1)}$, also when $q+a=s^2$.
}

		\smallskip \noindent\textbf{Induction step $k \to k+1$:}
		Assume the statement holds for $r = k$. For $r = k+1$, let $0 < L < 1 - {q^{1/(k+2)}}$. By the induction hypothesis, {$g_n-\mathrm{Id}\to g-\mathrm{Id}$ in $C^k$}. Setting $v_n := D^{k+1} g_n$ and differentiating $g_{n+1} = T(g_n)$ successively $k+1$ times via Fa\`{a} di Bruno's formula, the highest-order derivative satisfies
		\[
		v_{n+1}(x) = \mathcal{M}(g_n)(x) \cdot v_n(g_n^{-1}(x)) \cdot [Dg_n^{-1}(x)]^{\otimes(k+1)} + \mathcal{R}_k(g_n)(x),
		\]
		where $\mathcal{M}(g)(x) = [M + M D\phi_1(g^{-1}(x)) - \Lambda_1 D\phi_2(g^{-1}(x))] [Dg(g^{-1}(x))]^{-1}$ and $\mathcal{R}_k(g)$ depends only on the derivatives of $g$ up to order $k$.
		
		Define the fiber operator
		\[
		F_{k+1}(v, g)(x) := \mathcal{M}(g)(x) \cdot v(g^{-1}(x)) \cdot [Dg^{-1}(x)]^{\otimes(k+1)} + \mathcal{R}_k(g)(x).
		\]
		Using $\|Dg^{-1}\|_{C^0} \le {s^{-1}}$ and $\|\mathcal{M}(g)\|_{C^0} \le {(q+a)/s}$, the contraction factor satisfies
		\[
		\rho_{k+1} := {\frac{q+a}{s^{k+2}}} < 1,
		\]
		by the second condition in $B_\phi^{(k+1)}$.
{
To see the origin of the leading coefficient, temporarily write $g=g_n$
for a $C^{k+1}$ iterate, and put $h=g^{-1}$ and $m=k+1\ge2$.
Differentiating $g\circ h=\mathrm{Id}$ $m$ times gives
\[
Dg(h)D^m h+D^m g(h)[Dh]^{\otimes m}
+\mathcal R^{\mathrm{inv}}_{m-1}=0,
\]

Here the superscript $\mathrm{inv}$ labels the remainder arising from
the inverse relation $g\circ h=\mathrm{Id}$; it does not denote an
inverse operation. More precisely, $\mathcal R^{\mathrm{inv}}_{m-1}$
is the sum of the remaining Fa\`a di Bruno terms, in which the order
of the derivative of $g$ lies between $2$ and $m-1$. Each such term
has the form $D^j g(h)[D^{i_1}h,\ldots,D^{i_j}h]$, with its
multilinear arguments distributed among the factors, where
$2\le j\le m-1$, $i_1,\ldots,i_j\ge1$, and
$i_1+\cdots+i_j=m$. Thus only derivatives of $g$ and $h$ through
order $m-1$ occur. When $m=2$, the sum is empty and
$\mathcal R^{\mathrm{inv}}_1=0$. The dependence on $g$, $h$, and
$x$ is suppressed in this notation.

Multiplication by $[Dg(h)]^{-1}=Dh$ gives
\[
D^m h=-Dh\,D^m g(h)[Dh]^{\otimes m}
-Dh\,\mathcal R^{\mathrm{inv}}_{m-1}.
\]
Writing $\chi=\Lambda_1\phi_2-M\phi_1$, the term involving $D^m h$
in $D^mT(g)$ is $(D\chi(h)-M)D^m h$.
Substitution yields the stated leading coefficient
$\mathcal M(g)=(M-D\chi(h))Dh$.
All remaining inverse derivatives can be expressed recursively
using derivatives of $g$ through order $k$.
Here the variable $v$ ranges over the Banach space of continuous
periodic symmetric $(k+1)$-linear maps from $(\mathbb R^d)^{k+1}$
to $\mathbb R^d$; symmetry refers to the input variables.
The remainder $\mathcal R_k$ involves derivatives of the
perturbations up to order $k+1$, and derivatives of $g_n$ only up
to order $k$.
Uniform invertibility of $Dg_n$ and $C^k$ convergence of
$g_n-\mathrm{Id}$ ensure that $F_{k+1}(v,g_n)$ converges uniformly
to $F_{k+1}(v,g)$ for each fixed continuous field $v$.
Thus the parameter-continuity hypothesis of
Lemma~\ref{lem:convergence_of_derivatives} is satisfied.
}
Since {$g_n-\mathrm{Id}\to g-\mathrm{Id}$ in $C^k$}, Lemma~\ref{lem:convergence_of_derivatives} applies and yields $v_n = D^{k+1} g_n \to v$ uniformly. Thus $g \in L_L^{(k+1)}$.
		
		\medskip \noindent\textit{Step 3: Fractional regularity $r = k + \beta$ ($k \ge 1, \beta \in (0, 1)$).}\par
		Let $0 < L < 1 - {q^{1/(k+1+\beta)}}$. By the integer-order results, {$g_n-\mathrm{Id}\to g-\mathrm{Id}$ in $C^k$}. We now prove that $[D^k g]_\beta < \infty$.
		
		Let $v_n := D^k g_n$ and $H_n := [v_n]_\beta$. Differentiating the recurrence $g_{n+1} = T(g_n)$, we evaluate the difference $v_{n+1}(x) - v_{n+1}(y)$ for any distinct points $x, y \in \mathbb{R}^d$:

			\noindent\textbf{Case 1: $k = 1$.}
			For $k=1$, \begin{align*}v_{n+1}(x) = (I_d + M) + D\phi_1(x) + \big[\Lambda_1 D\phi_2(g_n^{-1}(x)) - M D\phi_1(g_n^{-1}(x)) - M\big] [v_n(g_n^{-1}(x))]^{-1}.\end{align*} Writing the difference as
			\begin{align*}
				&v_{n+1}(x) - v_{n+1}(y)\\
 =& D\phi_1(x) - D\phi_1(y) + \big[\Lambda_1 D\phi_2(g_n^{-1}(x)) - M D\phi_1(g_n^{-1}(x)) - M\big] \Big( [v_n(g_n^{-1}(x))]^{-1} - [v_n(g_n^{-1}(y))]^{-1} \Big) \\
				&\quad + \Big( [\Lambda_1 D\phi_2(g_n^{-1}(x)) - M D\phi_1(g_n^{-1}(x))] - [\Lambda_1 D\phi_2(g_n^{-1}(y)) - M D\phi_1(g_n^{-1}(y))] \Big) [v_n(g_n^{-1}(y))]^{-1},
			\end{align*}
			and applying the matrix identity $A^{-1} - B^{-1} = -A^{-1}(A - B)B^{-1}$ to linearize the inverse:
			\[
			[v_n(g_n^{-1}(x))]^{-1} - [v_n(g_n^{-1}(y))]^{-1} = - [v_n(g_n^{-1}(x))]^{-1} \big( v_n(g_n^{-1}(x)) - v_n(g_n^{-1}(y)) \big) [v_n(g_n^{-1}(y))]^{-1},
			\]
			the norm of the second term is bounded by
			\[
			{\frac{q+a}{s^2}} \|v_n(g_n^{-1}(x)) - v_n(g_n^{-1}(y))\|.
			\]
			Dividing by $|x - y|^\beta$ and using the coordinate stretching
			\[
			\frac{|g_n^{-1}(x) - g_n^{-1}(y)|^\beta}{|x - y|^\beta} \le \|g_n^{-1}\|_{\mathrm{Lip}}^\beta \le {s^{-\beta}},
			\]
			this leading term yields
			\[
			{\frac{q+a}{s^2}} \cdot {s^{-\beta}} H_n = {\frac{q+a}{s^{2+\beta}}} H_n = \rho_{1, \beta} H_n.
			\]
			Since $\phi_1, \phi_2 \in C^{1, \beta}$ and $\|v_n^{-1}\|_{C^0} \le {s^{-1}}$, the remaining terms divided by $|x - y|^\beta$ are uniformly bounded by a constant independent of $n$.

			\noindent\textbf{Case 2: $k \ge 2$.}
			Applying the Fa\`a di Bruno formula to isolate the highest-order derivative term, we write
			
			\[
			v_{n+1}(x) = \mathcal{M}(g_n)(x) \cdot v_n(g_n^{-1}(x)) \cdot \bigl[Dg_n^{-1}(x)\bigr]^{\otimes k} + \mathcal{R}_{k-1}(g_n)(x).
			\]
			Decomposing the difference of the leading product as
			\begin{align*}
				&\mathcal{M}(g_n)(x) \cdot v_n(g_n^{-1}(x)) \cdot [Dg_n^{-1}(x)]^{\otimes k} - \mathcal{M}(g_n)(y) \cdot v_n(g_n^{-1}(y)) \cdot [Dg_n^{-1}(y)]^{\otimes k} \\
				&\quad = \mathcal{M}(g_n)(x) \cdot \big( v_n(g_n^{-1}(x)) - v_n(g_n^{-1}(y)) \big) \cdot [Dg_n^{-1}(x)]^{\otimes k} \\
				&\qquad + \big( \mathcal{M}(g_n)(x) - \mathcal{M}(g_n)(y) \big) \cdot v_n(g_n^{-1}(y)) \cdot [Dg_n^{-1}(x)]^{\otimes k} \\
				&\qquad + \mathcal{M}(g_n)(y) \cdot v_n(g_n^{-1}(y)) \cdot \Big( [Dg_n^{-1}(x)]^{\otimes k} - [Dg_n^{-1}(y)]^{\otimes k} \Big),
			\end{align*}
			the first term divided by $|x - y|^\beta$ is bounded by
			\begin{align*}
			&\|\mathcal{M}(g_n)\|_{C^0} \|Dg_n^{-1}\|_{C^0}^k \frac{\|v_n(g_n^{-1}(x)) - v_n(g_n^{-1}(y))\|}{|g_n^{-1}(x) - g_n^{-1}(y)|^\beta} \left( \frac{|g_n^{-1}(x) - g_n^{-1}(y)|}{|x - y|} \right)^\beta\\
 \le & {\frac{q+a}{s^{k+1+\beta}}} H_n = \rho_{k, \beta} H_n.
			\end{align*}
			
For $k\ge2$, uniform $C^k$ bounds on $g_n-\mathrm{Id}$ give uniform Lipschitz bounds on derivatives of $g_n$ through order $k-1$ and on the corresponding derivatives of $g_n^{-1}$. The perturbations belong to $C^{k,\beta}$, so their derivatives of order $k$ have bounded $\beta$-H\"older seminorms. Thus the coefficient differences in the product decomposition and all differences of the remainder $\mathcal R_{k-1}(g_n)$, divided by $|x-y|^\beta$, are uniformly bounded independently of $n$.

			{In both cases, the H\"older seminorms satisfy}
			\[
			H_{n+1} \le \rho_{k, \beta} H_n + {C_H}, \quad \text{with} \quad \rho_{k, \beta} := {\frac{q+a}{s^{k+1+\beta}}},
			\]
			{where $C_H>0$ is independent of $n$.} By the second condition in ${B_\phi^{(r)}}$, $\rho_{k, \beta} < 1$. Since $g_0 = \mathrm{Id}$ is smooth ($H_0 = 0$), iterating the recurrence yields $H_n \le \frac{{C_H}}{1 - \rho_{k, \beta}} =: K < \infty$ for all $n$. Passing to the limit gives $[D^k g]_\beta \le K < \infty$. Hence $g \in {L_L^{(r)}}$.

	\begin{proof}[Proof of Theorem~\ref{thm:graph_persistence}]
The claim follows by Theorem~\ref{thm:DSM_Cr_smoothness} that ensure that a fixed point $g\in L_L^{(r)}$ exists, together with Proposition~\ref{pro:Herman}, proving that the function in
\eqref{eq:psi_explicit_formula} defines a $C^r$ invariant graph.
The hypotheses imply $L<1-\|M\|$ and
$(\phi_1,\phi_2)\in B_\phi^{(0)}$.
Lemma~\ref{lem:Lipschitz-I-G} and its proof give exponential attraction,
uniformly on bounded sets. Its uniqueness argument also excludes
any other continuous invariant graph.
\end{proof}

\subsection{A sufficient bound on the two perturbations}
\begin{Corollary}\label{cor:thm1}
		Let $\delta = 1 - \|M\|$. If $\|\Lambda_1\| < 1$, then for any real number $r \ge 1$, provided that
		\begin{equation}
			\max\{\|\phi_1\|_{C^1}, \|\phi_2\|_{C^1}\} < \frac{-B + \sqrt{B^2 + 16r\delta^2}}{8},
		\end{equation}
		where $B = 3(r+1)^2 + 2(r-1)\delta$, the map \eqref{eq:dissipative_map} admits a unique $C^r$-smooth invariant graph.
	\end{Corollary}
	
	\begin{proof}
		Denote $m := \max\{\|\phi_1\|_{C^1}, \|\phi_2\|_{C^1}\}$. Since $\|\Lambda_1\| < 1$ and $\|M\| < 1$, we have $\|\Lambda_1 \phi_2 - M \phi_1\|_{C^1} {\le} 2m$. According to Theorem~\ref{thm:DSM_Cr_smoothness}, it suffices to find $L \in (0, 1 - \|M\|^{\frac{1}{r+1}})$ satisfying $3m \le L(\delta - L)$ and $2m \le \delta - (r+1)L$. Maximizing the allowable bound $m$ leads to the optimization problem:
		\[
		m^* := \max_{L} \min \left\{ \frac{L(\delta - L)}{3}, \; \frac{\delta - (r+1)L}{2} \right\}.
		\]
		
The optimization is taken over $0\le L\le\delta/(r+1)$. On this interval the first function is increasing and the second is decreasing, so the maximum occurs at their unique intersection, where $L^*=(\delta-2m^*)/(r+1)$. Moreover,
\[
L^*<\frac{\delta}{r+1}<1-\|M\|^{1/(r+1)},\qquad
(1-L)^{r+1}\ge1-(r+1)L.
\]
For $m<m^*$, choose $L=L^*$. Then $3m<3m^*=L^*(\delta-L^*)$ and
\[
\|\Lambda_1\phi_2-M\phi_1\|_{C^1}
\le2m<2m^*=\delta-(r+1)L^*
\le(1-L^*)^{r+1}-\|M\|.
\]
Moreover, $\|\phi_1\|_{C^1}(1-L^*)+
\|\Lambda_1\phi_2-M\phi_1\|_{C^1}\le3m$.
Thus both conditions of Theorem~\ref{thm:DSM_Cr_smoothness} hold,
including their required strict form for $r>1$.
 Substituting this into \[3m^* = L^*(\delta - L^*)\] yields the quadratic equation
		\[
		4(m^*)^2 + \big[3(r+1)^2 + 2(r-1)\delta\big]m^* - r\delta^2 = 0.
		\]
		Solving for the positive root yields $m^* = \frac{-B + \sqrt{B^2 + 16r\delta^2}}{8}$.
	\end{proof}

\section{An example of a non-Lagrangian graph}\label{sec:non_lagrangian}
For $d\ge2$, fix $\lambda\in(0,1)$ and $\omega\in\mathbb R^d$.
Consider
\[
 F^{\phi_\varepsilon}(x,y)
 :=\bigl(x+\omega+y,\,\lambda y+\phi_\varepsilon(x)\bigr),\qquad
 \phi_\varepsilon(x)
 :=\varepsilon\bigl(\sin(2\pi x_2),\sin(2\pi x_1),0,\ldots,0\bigr)^T.
\]
For every finite real $r\ge1$, we will choose $\varepsilon_0>0$ such
that, for $0<\varepsilon<\varepsilon_0$, this is a dissipative twist map
with a unique $C^r$ invariant graph, and the graph is not Lagrangian.
The maps are real analytic and converge to the integrable map as
$\varepsilon\to0$ on each fixed complex strip. Thus the Lagrangian
property can fail under arbitrarily small analytic perturbations.
We have $\|\phi_\varepsilon\|_{C^0}=\sqrt2\varepsilon$ and
$\|D\phi_\varepsilon\|_{C^0}=2\pi\varepsilon$.

	\medskip \noindent\textit{Step 1: Dissipative twist property and non-conformal symplecticity.}\par
	 Writing $F^{\phi_\varepsilon}(x, y)=(X,Y)$,
	since $\partial_y X(x, y) = I_d$ is non-degenerate, the twist condition  holds. By the block determinant formula,
	\[
	\begin{aligned}
 \det DF^{\phi_\varepsilon}(x,y)
 &=\det\begin{pmatrix}I_d&I_d\\D\phi_\varepsilon(x)&\lambda I_d\end{pmatrix}
 =\det(\lambda I_d-D\phi_\varepsilon(x))\\
 &=\lambda^{d-2}\big[\lambda^2-(2\pi\varepsilon)^2
       \cos(2\pi x_1)\cos(2\pi x_2)\big].
\end{aligned}
	\]
	Whenever $\varepsilon < \frac{\lambda}{2\pi}$ and $\varepsilon < \frac{\sqrt{1-\lambda^2}}{2\pi}$, the determinant satisfies the uniform bounds
	\[
	0 < \lambda^{d-2}(\lambda^2 - (2\pi\varepsilon)^2) \le \det DF^{\phi_\varepsilon}(x, y) \le \lambda^{d-2}(\lambda^2 + (2\pi\varepsilon)^2) < 1.
	\]
	To verify global invertibility, let $(X, Y) = F^{\phi_\varepsilon}(x, y)$, eliminating $y = X - x - \omega$ yields $\lambda x - \phi_\varepsilon(x) = \lambda(X - \omega) - Y$. Since $\|D\phi_\varepsilon\|_{C^0} < \lambda$, the map $x \mapsto \frac{1}{\lambda}(\phi_\varepsilon(x) + \lambda(X-\omega) - Y)$ is a strict contraction on $\mathbb{R}^d$. By the Banach fixed-point theorem, $x$ and consequently $y$ are uniquely determined. An explicit degree-one isotopy to the identity is given by $F_t(x, y) = \big(x + t(\omega + y), \; (1-t+t\lambda)y + t\phi_\varepsilon(x)\big)$ for $t \in [0, 1]$. The map is a dissipative twist map.
	
	Furthermore, {pulling back the symplectic form $\Omega$ gives}
	\begin{align*}
		&(F^{\phi_\varepsilon})^*\Omega - \lambda \Omega\\
		=& \sum_{i=1}^d (dx_i + dy_i) \wedge (\lambda dy_i + d\phi_{\varepsilon, i}) - \lambda \sum_{i=1}^d dx_i \wedge dy_i \\
		= &2\pi\varepsilon \big(\cos(2\pi x_2) - \cos(2\pi x_1)\big) dx_1 \wedge dx_2 + 2\pi\varepsilon \cos(2\pi x_2)\, dy_1 \wedge dx_2 + 2\pi\varepsilon \cos(2\pi x_1)\, dy_2 \wedge dx_1.
	\end{align*}
	Because terms of type $dy_i \wedge dx_j$ are non-vanishing on an open set, $(F^{\phi_\varepsilon})^*\Omega$ cannot be proportional to $\Omega$. Thus $F^{\phi_\varepsilon}$ is not conformally symplectic for any factor.
	
	\medskip \noindent\textit{Step 2: Persistence and regularity of the invariant graph.}\par
	The map fits into model \eqref{eq:dissipative_map} with $\Lambda_1 = I_d$, $\Lambda_2 = \lambda I_d$, $\alpha_1 = \omega$, $\alpha_2 = 0$, $\phi_1 = 0$, and $\phi_2 = \phi_\varepsilon$. Here $M = \lambda I_d$, $\|M\| = \lambda < 1$, and $\Lambda_1 \phi_2 - M \phi_1 = \phi_\varepsilon$.
	For $0 < L < 1 - \lambda^{\frac{1}{r+1}}$, the smallness conditions of Theorem~\ref{thm:DSM_Cr_smoothness} reduce to
	\[
	2\pi\varepsilon \le L(1 - L - \lambda) \quad \text{and} \quad 2\pi\varepsilon < (1-L)^{r+1} - \lambda.
	\]
	Under these conditions, Theorem~\ref{thm:DSM_Cr_smoothness} yields a unique fixed point $g \in L_L^{(r)}$, defining a unique $C^r$-smooth invariant graph $\Gamma_\psi = \{(x, \psi(x)) \mid x \in \mathbb{R}^d\}$ with $\psi(x) = g(x) - x - \omega$ and $\|D\psi\|_{C^0} \le L$. Its projection $\bar{\Gamma}_\psi$ is the global attractor.
	
	\medskip \noindent\textit{Step 3: Destruction of the Lagrangian property.}\par
	Differentiating the invariance relation $\psi(g(x)) = \lambda \psi(x) + \phi_\varepsilon(x)$ yields
	\[
	D\psi(g(x)) \big(I_d + D\psi(x)\big) = \lambda D\psi(x) + D\phi_\varepsilon(x).
	\]
	Since $\|D\psi\|_{C^0} \le L < 1$, $I_d + D\psi(x)$ is invertible with $\|(I_d + D\psi(x))^{-1}\| \le \frac{1}{1-L}$. Taking the supremum over $x \in \mathbb{R}^d$ and using the surjectivity of the diffeomorphism $g$, we obtain
	\[
	\|D\psi\|_{C^0} \le \frac{\lambda \|D\psi\|_{C^0} + 2\pi\varepsilon}{1-L} \implies \|D\psi\|_{C^0} \le \frac{2\pi\varepsilon}{1-L-\lambda} =: C_{\lambda, L} \varepsilon.
	\]
	Now, assume for contradiction that $\Gamma_\psi$ is Lagrangian, so that $D\psi(x) = (D\psi(x))^T$ for all $x$. Transposing the differentiated invariance relation and subtracting the transpose yields the commutator identity:
	\begin{equation}\label{eq:commutator_identity}
		D\psi(g(x)) D\psi(x) - D\psi(x) D\psi(g(x)) = D\phi_\varepsilon(x) - \big(D\phi_\varepsilon(x)\big)^T.
	\end{equation}
	Evaluating the $(1,2)$-entry of the right-hand side at $x_0 = (0, \frac{1}{2}, 0, \dots, 0)^T$ gives
	\[
	\left| \big[ D\phi_\varepsilon(x_0) - (D\phi_\varepsilon(x_0))^T \big]_{12} \right| = 2\pi\varepsilon |\cos(\pi) - \cos(0)| = 4\pi\varepsilon.
	\]
	On the other hand, estimating the $(1,2)$-entry of the commutator on the left-hand side via the operator norm yields
	\[
	\left| \big[ D\psi(g(x_0)) D\psi(x_0) - D\psi(x_0) D\psi(g(x_0)) \big]_{12} \right| \le 2 \|D\psi\|_{C^0}^2 \le 2 C_{\lambda, L}^2 \varepsilon^2.
	\]
	Combining these estimates gives $4\pi\varepsilon \le 2 C_{\lambda, L}^2 \varepsilon^2$, which forces $\varepsilon \ge \frac{2\pi}{C_{\lambda, L}^2}$. Defining
	\begin{equation}\label{eq:eps_threshold_final}
		\varepsilon_0 := \min\left\{ \frac{\lambda}{2\pi}, \; \frac{\sqrt{1-\lambda^2}}{2\pi}, \; \frac{L(1-L-\lambda)}{2\pi}, \; \frac{(1-L)^{r+1}-\lambda}{2\pi}, \; \frac{2\pi}{C_{\lambda, L}^2} \right\} > 0,
	\end{equation}
	For all $0 < \varepsilon < \varepsilon_0$, {this contradicts $\varepsilon\ge 2\pi/C_{\lambda,L}^2$, so $\Gamma_\psi$ is not Lagrangian.} This proves the claimed properties of the example.

	\vspace{1em}

\section{Exact conformally symplectic twist maps}\label{sec:ecs}
{\color{black}
Theorem~\ref{thm:main_lagrangian} combines the persistence result of
Theorem~\ref{thm:graph_persistence} with the exact conformally symplectic
identities. We prove $C^1$ persistence and the exact Lagrangian property,
including the endpoint, in Theorem~\ref{thm:Lagrangian}, and higher
regularity in Theorem~\ref{thm:r-map}. After discussing sharpness, we
conclude this section with the proof of Theorem~\ref{thm:main_lagrangian}.
\par}
We consider the ECS lift $F:=F^{\Phi}_{\lambda,\alpha}$ from
\eqref{eq:conformal_symplectic_twist_map}, with $0<\lambda<1$ and a periodic
potential $\Phi$.

\subsection{Exact structure and $C^1$ persistence}\label{sec:ecs_c1}
Write $F(x,y)=(X,Y)$, so that
\[
Y=\lambda y+D\Phi(x),\qquad X=x+\alpha+Y.
\]
For $\Phi\in C^2(\mathbb T^d,\mathbb R)$, the generating function is
\begin{equation}\label{eq:generating_function}
S_{\alpha,\Phi}(x,X):=\frac12|X-x|^2-\langle\alpha,X-x\rangle+\Phi(x).
\end{equation}
It satisfies
\[
\lambda y=-\partial_x S_{\alpha,\Phi}(x,X),\qquad
Y=\partial_X S_{\alpha,\Phi}(x,X).
\]
When $\Phi=0$, we write $S_\alpha$ for $S_{\alpha,0}$.
In terms of the canonical forms $\vartheta$ and $\Omega$ defined in
Section~\ref{sec:notation}, the primitive relation is
\begin{equation}\label{eq:primitive_relation}
F^*\vartheta-\lambda\vartheta=dG,
\end{equation}
where
\begin{equation}\label{eq:primitive_G}
G(x,y):=\Phi(x)+\frac12\left|\lambda y+D\Phi(x)\right|^2.
\end{equation}
The function $G$ is $C^1$ and periodic in $x$, so it descends to the
cylinder. These identities verify the exact conformally symplectic
condition of Definition~\ref{def:Conformally Symplectic Twist Map}.

For general conformally symplectic maps, invariant submanifolds need not
be isotropic. Arnaud and F\'ejoz \cite[Proposition~3]{AF24} prove that a
two-dimensional $C^1$ closed invariant surface is isotropic, hence
Lagrangian when $d=2$. In higher dimensions their Proposition~4 gives
isotropy under an additional entropy condition. For
\eqref{eq:conformal_symplectic_twist_map}, the following result uses the
formula for $DF$ and requires no additional assumption on the dynamics
within the graph.

\begin{Theorem}\label{thm:Lagrangian}
		Let \(F^{\Phi}_{\lambda,\alpha}\) be the ECS twist map defined in \eqref{eq:conformal_symplectic_twist_map}. Assume \(\Phi \in C^2(\mathbb{T}^d, \mathbb{R})\) satisfies
		\begin{equation}\label{critc}
			{\|D\Phi\|_{C^1} \le (1 - \sqrt{\lambda})^2.}
		\end{equation}
		Then \(F^{\Phi}_{\lambda,\alpha}\) admits a unique \(C^1\) invariant exact Lagrangian graph.
	\end{Theorem}

\begin{Remark}
For the ECS model \eqref{eq:conformal_symplectic_twist_map}, \cite[Proposition~3.2]{SW26} proves that every \(C^0\) Lagrangian invariant graph is Lipschitz. Theorem~\ref{thm:Lagrangian} gives the stronger conclusion of a \(C^1\) exact Lagrangian invariant graph throughout the range \eqref{critc}. For \(d=1\), its \(C^1\) existence conclusion agrees with \cite[Theorem~1]{HLW26}; the proof below treats all dimensions.
\end{Remark}

\begin{proof}[Proof of Theorem~\ref{thm:Lagrangian}]

\medskip\noindent\textit{Step 1: Existence and uniqueness of the $C^1$ invariant graph.}\par
We apply Theorem~\ref{thm:DSM_Cr_smoothness} with $r=1$ and
\[
\Lambda_1=\Lambda_2=\lambda I_d,\quad
\alpha_1=\alpha,\quad\alpha_2=0,\quad
\phi_1=\phi_2=D\Phi.
\]
Thus $M=\lambda I_d$, $C=(1-\lambda)\alpha$, and
\[
\Lambda_1\phi_2-M\phi_1\equiv0,\qquad
T(g)(x)=(1+\lambda)x+(1-\lambda)\alpha+D\Phi(x)-\lambda g^{-1}(x).
\]
Set $L=1-\sqrt\lambda$. The first inequality in \eqref{eq:persistence_smallness} becomes
\[
\|D\Phi\|_{C^1}\le h(L):=\frac{L(1-L-\lambda)}{1-L}
=h(1-\sqrt\lambda)=(1-\sqrt\lambda)^2,
\]
and the second inequality is $0\le(1-L)^2-\lambda=0$.
Both hold by \eqref{critc}. Hence there is a unique $C^1$ fixed point $g$ with
\begin{equation}\label{eq:ECS_endpoint_base_bound}
\|Dg-I_d\|_{C^0}\le1-\sqrt\lambda,
\qquad{\|Dg^{-1}\|_{C^0}}\le\lambda^{-1/2}.
\end{equation}
{By Lemma~\ref{lem:Lipschitz-I-G}, the corresponding graph $\Gamma_\psi$ is unique among continuous invariant graphs, and its projection $\bar\Gamma_\psi$ is the global attractor.} In this model the graph function also satisfies
\begin{equation}\label{eq:ECS_graph_inverse_formula}
\psi(x)=x-\alpha-g^{-1}(x),\qquad
\|D\psi\|_{C^0}\le\frac{L}{1-L}
=\frac{1-\sqrt\lambda}{\sqrt\lambda}.
\end{equation}

\medskip\noindent\textit{Step 2: Lagrangian property.}\par
{We use the primitive relation \eqref{eq:primitive_relation}, where $G$ is periodic in $x$ and belongs to $C^1$ because $\Phi\in C^2$.}
The symmetry of $D^2\Phi$ gives directly
\begin{equation}\label{eq:conformal_Omega}
F^*\Omega=\sum_i(dx_i+dY_i)\wedge dY_i
=\lambda\Omega.
\end{equation}
This computation avoids differentiating a merely continuous pullback 1-form.
{Using the graph parameterization $\iota(x)=(x,\psi(x))$ from Proposition~\ref{pro:Herman}, set}
\[
\eta:=\iota^*\vartheta=\psi\cdot dx,\qquad
\omega_\psi:=\iota^*\Omega=-d\eta.
\]
{We identify these periodic forms with their descents to $\mathbb T^d$, and write $\bar g$ for the torus map induced by $g$. The identity $F\circ\iota=\iota\circ g$ implies}
\[
{\bar g^*\omega_\psi=\lambda\omega_\psi.}
\]
In the strict range, choosing an admissible $L<1-\sqrt\lambda$ in Step~1, the norm estimate gives
\(\|\omega_\psi\|_{C^0}\le\lambda{\|Dg^{-1}\|_{C^0}^2}
\|\omega_\psi\|_{C^0}\) with a coefficient smaller than one.
At the endpoint we use the equality structure of the same pullback relation.
Put
\[
P(x):=D\psi(x),\quad W(x):=P(x)-P(x)^T,\quad A(x):=Dg(x).
\]
The matrix $W$ represents $\omega_\psi$ in the convention
$\omega_\psi(x)[u,v]=u^TW(x)v$ for tangent vectors $u,v$.
{Since $A=I_d+\lambda P+D^2\Phi$, the symmetry of $D^2\Phi$ and the pullback identity give}
\[
A-A^T=\lambda W,\qquad A(x)^TW(g(x))A(x)=\lambda W(x).
\]
Consequently,
\begin{equation}\label{eq:ECS_skew_pullback}
W(g(x))=\lambda A(x)^{-T}W(x)A(x)^{-1}
=A(x)^{-T}-A(x)^{-1}.
\end{equation}
In the last equality we used
$A^{-T}(A-A^T)A^{-1}=A^{-T}-A^{-1}$.
By \eqref{eq:ECS_endpoint_base_bound} and
$\|A^T-I_d\|=\|A-I_d\|$, $A$ and $A^T$ belong to the matrix ball in Lemma~\ref{lem:strict_inverse_endpoint} with $1-L=\sqrt\lambda$.
Whenever $W(x)\ne0$, that lemma yields
\[
\|W(g(x))\|=\|A^{-T}-A^{-1}\|
<\lambda^{-1}\|A-A^T\|=\|W(x)\|.
\]

If $W\not\equiv0$, continuity and periodicity give a point $x_*$ at which $\|W\|$ attains its positive maximum. Set $x=g^{-1}(x_*)$. Equation~\eqref{eq:ECS_skew_pullback} ensures $W(x)\ne0$, so the strict inequality gives
\(\|W(x_*)\|<\|W(x)\|\le\|W(x_*)\|\), a contradiction.

Thus $W\equiv0$, so $\omega_\psi=0$ and the graph is Lagrangian.

\medskip\noindent\textit{Step 3: Exactness.}\par
Pulling back the primitive relation by $\iota$ gives
\[
{\bar g^*\eta-\lambda\eta=d(G\circ\iota).}
\]
Now $\eta$ is a closed $C^1$ 1-form. At this regularity its cohomology class can be read from its periods over closed curves. The $C^1$ homotopy
{\(g_t(x)=(1-t)x+t g(x)\) induces a homotopy $\bar g_t$ from the identity to $\bar g$ on the torus, since $g-\mathrm{Id}$ is periodic. For every closed $C^1$ curve $\gamma$ on $\mathbb T^d$, Stokes' theorem on the cylinder traced by $\bar g_t\circ\gamma$, together with $d\eta=0$, gives}
\[
{\int_\gamma\bar g^*\eta=\int_{\bar g\circ\gamma}\eta=\int_\gamma\eta.}
\]
Integration of the primitive relation gives
\((1-\lambda)\int_\gamma\eta=0\). Equivalently,
\begin{equation}\label{eq:cohomology_relation_step3}
{[\bar g^*\eta]}-\lambda[\eta]=(1-\lambda)[\eta]=0
\quad\text{in }H^1_{\mathrm{dR}}(\mathbb T^d).
\end{equation}
Since $0<\lambda<1$, all periods vanish. Define $U(x)$ by integrating
$\eta$ along a path from a fixed base point to $x$ on the torus.
Vanishing of all periods makes this definition independent of the path,
up to the chosen additive constant. It satisfies $dU=\eta$;
since the coefficients $\psi$ are $C^1$, we have $U\in C^2(\mathbb T^d)$.
Equivalently, the lift of $U$ is periodic and $\psi=DU$. This proves exactness.
\end{proof}

\begin{Remark}\label{rmk:weak_kam_variational}
		If {$\|D^2\Phi\|_{C^0} \le (1 - \sqrt{\lambda})^2$}, the weighted finite-interval action associated with the generating function $S_{\alpha,\Phi}$ in \eqref{eq:generating_function} is defined, for integers $m < n$, by
		\[
		\mathcal{A}_{m,n}(x_m, \ldots, x_n) := \sum_{k=m}^{n-1} \lambda^{-k} S_{\alpha, \Phi}(x_k, x_{k+1}).
		\]
		For variations with fixed endpoints $\xi_m = \xi_n = 0$, the second variation takes the form
		\[
		\delta^2 \mathcal{A}_{m,n} = \sum_{k=m}^{n-1} \lambda^{-k} \left( |\xi_{k+1} - \xi_k|^2 + \langle D^2\Phi(x_k)\xi_k, \xi_k \rangle \right).
		\]
		
With $a_k:=\lambda^{-k/2}\xi_k$, the endpoint conditions give $a_m=a_n=0$, and
\[
\delta^2\mathcal A_{m,n}
\ge\sum_{k=m}^{n-1}
\left(|\sqrt\lambda\,a_{k+1}-a_k|^2
-(1-\sqrt\lambda)^2|a_k|^2\right)
=\sqrt\lambda\sum_{k=m}^{n-1}|a_{k+1}-a_k|^2.
\]
The equality follows because the difference between the two sums is
$(\lambda-\sqrt\lambda)\sum_{k=m}^{n-1}(|a_{k+1}|^2-|a_k|^2)=0$.
If the final sum were zero, all $a_k$ would be equal; the endpoint
condition would then force every $a_k$ to vanish.
Thus it is strictly positive for every nonzero variation with fixed endpoints.  Orbit segments satisfy its Euler--Lagrange equations and are consequently the unique minimizers among lifted configurations with the same endpoints.

	\end{Remark}

\subsection{Higher regularity and its endpoint}\label{sec:ecs_higher}
Recall the threshold $\eta_r(\lambda)$ defined in \eqref{eq:eta_r}.

\begin{Theorem}\label{thm:r-map}
Let $\lambda\in(0,1)$ and $1\le r<\infty$. If
$\Phi\in C^{r+1}(\mathbb T^d,\mathbb R)$ satisfies
\[
\|D\Phi\|_{C^1}<\eta_r(\lambda),
\]
then $F^{\Phi}_{\lambda,\alpha}$ admits a unique $C^r$ invariant exact
Lagrangian graph. For $r=1$, the inequality may be replaced by $\le$,
by Theorem~\ref{thm:Lagrangian}.
\end{Theorem}

\begin{proof}[Proof of Theorem~\ref{thm:r-map}]

		{The case $r=1$, including equality, follows from Theorem~\ref{thm:Lagrangian}. For finite $r>1$, the cancellation $\Lambda_1\phi_2-M\phi_1\equiv0$ used in Step~1 of the proof of Theorem~\ref{thm:Lagrangian} reduces the conditions of Theorem~\ref{thm:DSM_Cr_smoothness} to finding $L$ with}
		\[
		0 < L < L_r := 1 - \lambda^{\frac{1}{r+1}}
		\]
		{such that $\|D\Phi\|_{C^1}\le h(L)$, with $h$ as defined there.}

Since $r>1$, we have $L_r<1-\sqrt\lambda$. The function $h$ is continuous on $[0,L_r]$ and strictly increasing there, so
		\[
\sup_{0<L<L_r}h(L)=h(L_r)
=\left(1-\lambda^{\frac1{r+1}}\right)
 \left(1-\lambda^{\frac r{r+1}}\right)=\eta_r(\lambda).
		\]
The hypothesis $\|D\Phi\|_{C^1}<\eta_r(\lambda)$ allows us to choose $L\in(0,L_r)$ with $\|D\Phi\|_{C^1}\le h(L)$. This choice also gives $(1-L)^{r+1}>\lambda$. Theorem~\ref{thm:DSM_Cr_smoothness} yields a unique $C^r$ invariant graph $\Gamma_\psi$, with
\[
\|Dg-I_d\|_{C^0}\le L,\qquad
\|D\psi\|_{C^0}\le\frac{L}{1-L},
\]
whose projection $\bar\Gamma_\psi$ is the global attractor.

Since $h(L)<h(1-\sqrt\lambda)=(1-\sqrt\lambda)^2$, Theorem~\ref{thm:Lagrangian} also applies. By uniqueness, it gives exactness and the Lagrangian property for the same graph.

		\end{proof}

The thresholds $\eta_r(\lambda)$ tend to zero as $r\to\infty$. Thus these estimates establish persistence for each finite $r\ge1$ but do not establish $C^\infty$ regularity for a fixed nonzero perturbation.

For later use, the threshold is strictly decreasing in $r>1$:
\[
\frac{d}{dr}\eta_r(\lambda)
=\frac{\log\lambda}{(r+1)^2}
\left(\lambda^{1/(r+1)}-\lambda^{r/(r+1)}\right)<0.
\]
Here $\log\lambda<0$, whereas the expression in parentheses is
positive for $r>1$.

\begin{Proposition}[Failure of the $C^2$ endpoint]\label{prop:C2_endpoint}
For every $\lambda\in(0,1)$ and every $d\geq1$, set
\[
\Phi(x):=\frac{\eta_2(\lambda)}{4\pi^2}\cos(2\pi x_1),
\qquad \alpha=0.
\]
Then $\Phi$ is real analytic and
\[
\|D\Phi\|_{C^1}=\eta_2(\lambda).
\]
The map $F^{\Phi}_{\lambda,0}$ has a unique continuous invariant graph,
which is exact Lagrangian and belongs to $C^{2-\varepsilon}$ for every
$\varepsilon\in(0,1)$, but is not of class $C^2$.
Consequently, the strict inequality in Theorem~\ref{thm:r-map} cannot
be replaced by a non-strict inequality when $r=2$, even for analytic maps.
\end{Proposition}

\begin{Remark}[Endpoints of higher regularity]\label{rem:higher_endpoint_counterexamples}
More generally, for every real $r>1$ and every $\lambda\in(0,1)$,
there are periodic potentials $\Phi\in C^{r+1}(\mathbb T^d,\mathbb R)$
with $\|D\Phi\|_{C^1}=\eta_r(\lambda)$ for which
$F^{\Phi}_{\lambda,0}$ has a unique invariant graph belonging to $C^q$
for every $1\leq q<r$, but not to $C^r$.
One construction starts in dimension one with the odd local perturbation
\[
\phi(x)=-\eta_r(\lambda)x+
\frac{c\,x|x|^{r-1}}{\log(e/|x|)},\qquad 0<|x|\ll1,
\qquad 0<c\le\eta_r(\lambda)/(4(r+1)),
\]
and $\phi(0)=0$.
Here the logarithmic correction is $C^r$ at the origin; its
derivatives through order $\lfloor r\rfloor$ vanish there (through
order $r$ when $r$ is an integer).
To describe a periodic extension, write $p_0$ for the local derivative
$\phi'$ and choose the neighborhood so that
$-\eta_r(\lambda)\le p_0\le0$.
Take a smooth even cutoff $0\le\chi\le1$, equal to one near zero,
supported in that neighborhood and satisfying
$\int_{\mathbb T}\chi<1/2$. Set
\[
b:=-\frac{\int_{\mathbb T}\chi p_0}{\int_{\mathbb T}(1-\chi)},
\qquad p:=\chi p_0+(1-\chi)b,\qquad
\phi(x):=\int_0^x p(t)\,dt.
\]
The product $\chi p_0$ is understood to be zero off the neighborhood.
Then $0<b<\eta_r(\lambda)$, $p$ is even with zero average, and
$p(0)=-\eta_r(\lambda)$ while $|p|\le\eta_r(\lambda)$.
Thus the last formula gives an odd periodic $C^r$ extension with
$\|\phi'\|_{C^0}=\eta_r(\lambda)$.
Oddness gives zero average and \(\|\phi\|_{C^0}\le\|\phi'\|_{C^0}/2\), so the full \(C^1\) norm is also \(\eta_r(\lambda)\). Its periodic primitive is of class $C^{r+1}$, and dependence on $x_1$
alone gives the examples in higher dimensions.
The obstruction is the resonance $\lambda=a^{r+1}$, where
$a=\lambda^{1/(r+1)}$ is the tangential multiplier at the origin:
the logarithmic term accumulates along orbits approaching the origin
and prevents $C^r$ regularity.

This construction shows that the $C^1$ threshold in
Theorem~\ref{thm:Lagrangian} does not extend to the corresponding
$C^r$ threshold for $r>1$; Proposition~\ref{prop:C2_endpoint}
provides an analytic example at $r=2$.
\end{Remark}

\begin{proof}[Proof of Proposition~\ref{prop:C2_endpoint}]
We first treat $d=1$.  Write
\[
a:=\lambda^{1/3},\qquad
\eta:=\eta_2(\lambda)=(1-a)(1-a^2),\qquad
\phi(x):=\Phi'(x)=-\frac{\eta}{2\pi}\sin(2\pi x).
\]
Since $\phi'(x)=-\eta\cos(2\pi x)$, the definition of the $C^1$ norm gives
$\|\phi\|_{C^1}=\max\{\eta/(2\pi),\eta\}=\eta$.
Moreover,
\[
(1-\sqrt\lambda)^2-\eta
=(1-a^{3/2})^2-(1-a)(1-a^2)
=a(1-\sqrt a)^2>0.
\]
Thus Theorem~\ref{thm:Lagrangian}
gives a unique $C^1$ invariant graph, denoted by $y=\psi(x)$, and
the induced lift
\[
g(x)=x+\lambda\psi(x)+\phi(x)
\]
is a $C^1$ diffeomorphism satisfying
\begin{equation}\label{eq:C2_endpoint_HM}
g(x)+\lambda g^{-1}(x)=(1+\lambda)x+\phi(x),
\qquad g'(x)\geq\sqrt\lambda>0.
\end{equation}
{The derivative bound follows from \eqref{eq:ECS_endpoint_base_bound}.}
The map $F^{\Phi}_{\lambda,0}$ commutes with $(x,y)\mapsto(-x,-y)$.
The reflected graph $y=-\psi(-x)$ is also a continuous
invariant graph. Uniqueness gives $\psi(-x)=-\psi(x)$, and hence
$g(-x)=-g(x)$. In particular, $g(0)=0$.
Also, $g(x+1)=g(x)+1$ gives
$g(1/2)=g(-1/2)+1=-g(1/2)+1$, so
\begin{equation}\label{eq:C2_endpoint_fixed_points}
g(0)=0,\qquad g(1/2)=1/2.
\end{equation}
Differentiating \eqref{eq:C2_endpoint_HM} at $0$ gives
\[
g'(0)+\frac{\lambda}{g'(0)}
=1+\lambda-\eta=a+a^2.
\]
The two roots are $a$ and $a^2$.
Since $a^2<\sqrt\lambda=a^{3/2}<a$ and
$g'(0)\geq\sqrt\lambda$, necessarily
\begin{equation}\label{eq:C2_endpoint_multiplier}
g'(0)=a<1.
\end{equation}

Suppose, for a contradiction, that $\psi\in C^2$.
Then $g\in C^2$, and $g^{-1}\in C^2$ because $g'\ge\sqrt\lambda>0$.
Put $q:=g''$, a continuous periodic function.
Differentiating the oddness identity twice gives $q(-x)=-q(x)$,
and in particular $q(0)=0$.
On the other hand, \eqref{eq:C2_endpoint_fixed_points} implies
\[
\int_0^{1/2}g'(x)\,dx=\frac12.
\]
As $g'(0)=a<1$, there is a point $w\in(0,1/2)$ with $g'(w)>a$.
The mean value theorem applied to $g'$ on $[0,w]$ gives a point
$z\in(0,w)$ with $q(z)>0$.

There is no fixed point of $g$ in $(0,1/2)$.
Indeed, if $g(x)=x$, invariance of the graph and the definition of $g$
give, respectively,
\[
\psi(x)=\lambda\psi(x)+\phi(x),\qquad
0=\lambda\psi(x)+\phi(x).
\]
Thus $\psi(x)=\phi(x)=0$, whereas $\phi(x)<0$ for $0<x<1/2$.
By \eqref{eq:C2_endpoint_multiplier}, $g(x)<x$ for small positive $x$;
continuity and the absence of other fixed points give
\[
0<g(x)<x<1/2\qquad (0<x<1/2).
\]
Consequently, $x_n:=g^n(z)$ decreases to a limit $\ell\in[0,z]$.
Continuity gives $g(\ell)=\ell$; the absence of fixed points in
$(0,1/2)$ forces $\ell=0$.
Moreover, for any $\tau_0\in(a,1)$, differentiability at $0$ gives
$g(x)\leq\tau_0 x$ for all sufficiently small positive $x$.
There is thus an $N$ such that $x_{N+j}\le\tau_0^j x_N$ for all
$j\ge0$. Hence $\sum_{n\ge0}x_n<\infty$.

Differentiating $g(g^{-1}(x))=x$ once and then twice gives
\[
(g^{-1})'(x)=\frac{1}{g'(g^{-1}(x))},\qquad
(g^{-1})''(x)=-\frac{q(g^{-1}(x))}{g'(g^{-1}(x))^3}.
\]
Twice differentiating \eqref{eq:C2_endpoint_HM} gives
$q(x)-\lambda q(g^{-1}(x))/g'(g^{-1}(x))^3=\phi''(x)$.
Setting $x=x_{n+1}$, so that $g^{-1}(x_{n+1})=x_n$, yields
\begin{equation}\label{eq:C2_endpoint_curvature}
q(x_{n+1})=\phi''(x_{n+1})+
\frac{\lambda}{g'(x_n)^3}q(x_n).
\end{equation}
Here $\phi''(x)=2\pi\eta\sin(2\pi x)>0$ on $(0,1/2)$.
Starting with $q(z)>0$, iteration of
\eqref{eq:C2_endpoint_curvature} gives
\begin{equation}\label{eq:C2_endpoint_product}
q(x_n)\geq q(z)
\prod_{j=0}^{n-1}\frac{\lambda}{g'(x_j)^3}.
\end{equation}
Under the assumed $C^2$ regularity,
\[
|g'(x_j)-a|\leq\|q\|_{C^0}x_j.
\]
The function $\log t$ has derivative at most $1/\sqrt\lambda$
on $[\sqrt\lambda,\infty)$, which contains both $a$ and every $g'(x_j)$.
Since $\lambda=a^3$, the mean value theorem gives
\[
\left|\log\frac{\lambda}{g'(x_j)^3}\right|
=3|\log a-\log g'(x_j)|
\le\frac{3}{\sqrt\lambda}|g'(x_j)-a|
\le\frac{3\|q\|_{C^0}}{\sqrt\lambda}x_j.
\]
Together with $\sum_jx_j<\infty$, this proves
\[
\sum_{j=0}^{\infty}
\left|\log\frac{\lambda}{g'(x_j)^3}\right|<\infty.
\]
Every factor in \eqref{eq:C2_endpoint_product} is positive, and
its partial product is the exponential of the corresponding partial
sum of logarithms. Absolute convergence of that series gives
\[
\prod_{j=0}^{n-1}\frac{\lambda}{g'(x_j)^3}
\longrightarrow
\exp\left(\sum_{j=0}^\infty\log\frac{\lambda}{g'(x_j)^3}\right)>0.
\]  It follows that
\[\liminf_{n\to\infty}q(x_n)>0,\] contradicting the continuity of $q$
and $q(x_n)\to q(0)=0$.
This proves $\psi\notin C^2$.
For every $\varepsilon\in(0,1)$,
$\eta_2(\lambda)<\eta_{2-\varepsilon}(\lambda)$;
Theorem~\ref{thm:r-map} gives
$\psi\in C^{2-\varepsilon}$.

For $d>1$, the chosen potential depends only on $x_1$.
The map is the direct product of the preceding one-dimensional map
and $d-1$ integrable maps $(x_j,y_j)\mapsto(x_j+\lambda y_j,\lambda y_j)$.
Its invariant graph is
\[
\psi_d(x_1,\ldots,x_d)=(\psi(x_1),0,\ldots,0).
\]
Uniqueness and exactness follow from Theorem~\ref{thm:Lagrangian},
and the asserted regularity and failure of $C^2$ follow by restriction
to the first coordinate.
\end{proof}

\subsection{Sharpness of the $C^1$ bound}\label{sec:ecs_sharpness}
The bound in \eqref{critc} concerns both the existence and the regularity
of the invariant graph. These lead to two different sharpness questions.

\paragraph{Regularity.}
Theorem~2 of \cite{HLW26} constructs nondifferentiable invariant graphs
at every prescribed larger norm. Thus \eqref{critc} is optimal as a
uniform guarantee of $C^1$ regularity over the perturbations and
parameters. The vertical drift in that one-dimensional result is removed
by the translation $y\mapsto y-\alpha_2/(1-\lambda)$, with the base drift
adjusted accordingly. Products with integrable maps give the same
obstruction in higher dimensions.

\paragraph{Existence as $\lambda\to1^-$.}
For persistence of Lipschitz Lagrangian graphs, one can also ask how the
allowable perturbation size scales with the dissipation. More precisely,
if $\|D\Phi\|_{C^1}=O((1-\lambda)^\gamma)$, which exponents $\gamma$
can give a uniform persistence statement? The following result is proved
in \cite{SW26}.

\begin{Proposition}\label{pro:trigonometric polynomials}
In dimension $d=1$, for every $\lambda\in(0,1)$ there exists a sequence
of trigonometric polynomials $\{\Phi_n^\lambda\}_{n\in\mathbb N}$ such
that the maps $F^{\Phi_n^\lambda}_{\lambda,\alpha}$ have no continuous
invariant graphs for any $\alpha\in\mathbb R$. Moreover, for every
$n\in\mathbb N$,
\[
\|D\Phi_n^\lambda\|_{C^1}\le1\quad(0<\lambda<1),\qquad
\|D\Phi_n^\lambda\|_{C^1}=O(1-\lambda)\quad(\lambda\to1^-).
\]
\end{Proposition}

Thus perturbations of order $O(1-\lambda)$ can destroy all invariant
graphs. This rules out a universal persistence statement based solely
on that order of magnitude, without an appropriate size restriction.
By contrast, Theorem~\ref{thm:Lagrangian} guarantees persistence at the
explicit size
\[
(1-\sqrt\lambda)^2\sim\frac{(1-\lambda)^2}{4}
\qquad(\lambda\to1^-).
\]

{\color{black}
\begin{proof}[Proof of Theorem~\ref{thm:main_lagrangian}]
Theorem~\ref{thm:Lagrangian} proves the case $r=1$, including equality,
and Theorem~\ref{thm:r-map} gives the $C^r$ exact Lagrangian graph for
$r>1$. Uniqueness among continuous periodic invariant graphs and global
attraction follow from Lemma~\ref{lem:Lipschitz-I-G} and its proof.
\end{proof}
}

\section{Regularity in dimension one}\label{sec:regularity}

{\color{black}
Theorem~\ref{thm:main_regularity} combines regularity of $C^1$ invariant
graphs with a counterexample showing that the $C^1$ hypothesis cannot
be replaced by Lipschitz continuity. We prove part~\textup{(i)} in
Theorem~\ref{thm:C_infty_smoothness} using derivative estimates and
fiber contraction, and construct the analytic counterexample for
part~\textup{(ii)} in Proposition~\ref{prop:irrational_Lip_not_C1}.
The proof of Theorem~\ref{thm:main_regularity} concludes this section.
\par}
We consider the map \eqref{eq:dissipative_map} in dimension one,
\begin{equation}\label{eq:1dDTM}
 F_{\lambda,\alpha}^{\phi}(x,y)
 =\bigl(x+\alpha_1+\lambda_1y+\phi_1(x),\,
        \alpha_2+\lambda_2y+\phi_2(x)\bigr),
\end{equation}
where $\lambda_1\ne0$, $\lambda_2\in(0,1)$,
$\alpha=(\alpha_1,\alpha_2)\in\mathbb R^2$, and
$\phi_1,\phi_2\in C^k(\mathbb T,\mathbb R)$ with $k\ge2$.
By \eqref{dissp} and periodicity (Remark~\ref{rem:det}), there are
constants $\delta>0$ and $\lambda\in(0,1)$ such that
\begin{equation}\label{eq:one_dimensional_dissipation}
 0<\delta\le\det DF_{\lambda,\alpha}^{\phi}(x,y)
 =\lambda_2-D(\lambda_1\phi_2-\lambda_2\phi_1)(x)
 \le\lambda<1.
\end{equation}
In fact, periodicity gives
$\lambda_2=\int_0^1\det DF_{\lambda,\alpha}^{\phi}(x,y)\,dx\in(0,1)$,
so this restriction on $\lambda_2$ follows from dissipation.
The invariant graph then has the same differentiability as the map,
provided that it is $C^1$ and its induced dynamics has irrational
rotation number.

\begin{Theorem}\label{thm:C_infty_smoothness}
Let $k\ge2$ be an integer. Suppose that the $C^k$ map
$F_{\lambda,\alpha}^{\phi}$ in \eqref{eq:1dDTM} satisfies
\eqref{eq:one_dimensional_dissipation} and preserves a graph
$\Gamma=\{(x,\psi(x)):x\in\mathbb R\}$ with
$\psi\in C^1(\mathbb T,\mathbb R)$. If the induced circle map has
irrational rotation number, then $\psi\in C^k(\mathbb T,\mathbb R)$.
In particular, if $F_{\lambda,\alpha}^{\phi}$ is $C^\infty$, then
$\psi\in C^\infty(\mathbb T,\mathbb R)$.
\end{Theorem}

We prove the theorem by applying the $C^r$ section theorem to the
slope equation. Unique ergodicity gives subexponential derivative
bounds, and dissipation gives the required fiber contraction.
Proposition~\ref{prop:irrational_Lip_not_C1} below shows why the
$C^1$ hypothesis is needed.

\subsection{Derivative estimates for the circle diffeomorphism}\label{sec:circle_derivatives}

Consider the map $F_{\lambda,\alpha}^{\phi}$ in \eqref{eq:1dDTM}
under the hypotheses of Theorem~\ref{thm:C_infty_smoothness}, first
with finite $k\ge2$. Let $\psi\in C^1(\mathbb T,\mathbb R)$ define
an invariant graph. Its induced base lift is

	\begin{equation}\label{eq:circle_homeomorphism}
		g(x) = x + \alpha_1 + \lambda_1 \psi(x) + \phi_1(x),
	\end{equation}
	{and invariance gives}
	\begin{equation}\label{eq:graph_transform}
		\psi(g(x)) = \alpha_2 + \lambda_2 \psi(x) + \phi_2(x).
	\end{equation}
	
	\begin{Lemma}\label{lem:g_circle_diffeo}
		The map $g$ defined in \eqref{eq:circle_homeomorphism} is a degree-one lift of an orientation-preserving $C^1$ circle diffeomorphism $\bar{g} \colon \mathbb{T} \to \mathbb{T}$ satisfying $D\bar{g}(\theta) > 0$ for all $\theta \in \mathbb{T}$.
	\end{Lemma}
	
\begin{proof}

By Proposition~\ref{pro:Herman}, $g$ is an increasing degree-one
homeomorphism. Equation~\eqref{eq:circle_homeomorphism} gives
$g\in\mathrm{Id}+C^1(\mathbb T,\mathbb R)$, so $Dg\ge0$.
Set $p=D\psi$. Differentiating \eqref{eq:circle_homeomorphism} and
\eqref{eq:graph_transform} yields
\[
 Dg(x)=1+\lambda_1p(x)+D\phi_1(x),\qquad
 p(g(x))Dg(x)=\lambda_2p(x)+D\phi_2(x).
\]
Consequently,
\begin{align*}
 \det DF_{\lambda,\alpha}^{\phi}(x,\psi(x))
 &=\lambda_2(1+D\phi_1(x))-\lambda_1D\phi_2(x)\\
 &=Dg(x)\bigl(\lambda_2-\lambda_1p(g(x))\bigr).
\end{align*}
The determinant is at least $\delta>0$, so $Dg$ cannot vanish.
Thus $\bar g$ is a $C^1$ circle diffeomorphism.

\end{proof}

Since $\bar g$ has irrational rotation number, it is uniquely ergodic;
write $\mu$ for its invariant probability measure. Applying the uniform
ergodic theorem to the continuous function $\log D\bar g$ gives
\[
 \frac1n\log D(\bar g^n)(\theta)
 =\frac1n\sum_{i=0}^{n-1}\log D\bar g(\bar g^i(\theta))
 \longrightarrow \chi:=\int_{\mathbb T}\log D\bar g\,d\mu
\]
uniformly in $\theta$. Put $P_n(\theta):=D(\bar g^n)(\theta)$
for $n\ge0$, so $P_0=1$. The degree-one property gives
$\int_{\mathbb T}P_n(\theta)\,d\theta=1$.
If $\chi>0$, uniform convergence would imply
$P_n\ge e^{\chi n/2}>1$ for all sufficiently large $n$; if $\chi<0$,
it would imply $P_n\le e^{\chi n/2}<1$. Both conclusions contradict
this integral identity. Hence $\chi=0$, and for every $\varepsilon>0$
there is $N_0(\varepsilon)\ge1$ such that, for all $n\ge N_0(\varepsilon)$
and $\theta\in\mathbb T$,
\begin{equation}\label{eq:derivative_exponential_bound}
 e^{-\varepsilon n}\le P_n(\theta)\le e^{\varepsilon n}.
\end{equation}

	\subsection{Construction of the local fiber map}

The derivative $p=D\psi$ descends to
$\bar p=D\bar\psi\in C^0(\mathbb T,\mathbb R)$.
Differentiating \eqref{eq:graph_transform} gives
\begin{equation}\label{eq:p_recursion_new}
 \bar p(\bar g(\theta))=\bar G(\theta,\bar p(\theta)),
\end{equation}
where $\bar G$ is defined in \eqref{eq:Phi_definition_new} below.
We will apply the following form of the $C^r$ section theorem
\cite[pp.~25--35]{HPS77}.

	\begin{Lemma}\label{lem:C^r section}
		Let $\bar{\mathcal{F}}(\theta,s) = \bigl(\bar{h}(\theta), \bar{G}(\theta,s)\bigr)$ be an $r$-fiber contraction on the trivial bundle $\mathbb{T} \times \mathbb{R}$ over the circle in the sense of Definition~\ref{def:r-fiber contraction}. Then there exists a unique continuous invariant section $\bar{p} \colon \mathbb{T} \to \mathbb{R}$ satisfying $\bar{p}(\bar{h}(\theta)) = \bar{G}(\theta, \bar{p}(\theta))$, and this section is of class $C^r$.
	\end{Lemma}

	{Define the slope map}
	\begin{equation}\label{eq:Phi_definition_new}
		\bar{G}(\theta, s) = \frac{\lambda_2 s + D\phi_2(\theta)}{1 + \lambda_1 s + D\phi_1(\theta)}.
	\end{equation}
	Since $D\bar{g}(\theta) > 0$, the denominator is strictly positive on the compact invariant slope section $\bar{\Gamma}_p = \{(\theta, \bar{p}(\theta)) \mid \theta \in \mathbb{T}\}$. There exists an open neighborhood $V$ of $\bar{\Gamma}_p$ in $\mathbb{T} \times \mathbb{R}$ on which the local fiber-preserving map $\bar{\mathcal{F}}(\theta, s) = (\bar{g}(\theta), \bar{G}(\theta, s))$ is well-defined and of class $C^1$.

The determinant along the graph is the periodic function
\[
 J(\theta):=\lambda_2(1+D\phi_1(\theta))-\lambda_1D\phi_2(\theta).
\]

	By the dissipative condition, $0 < J(\theta) \le \lambda < 1$. Differentiating $\bar{G}$ with respect to $s$ yields:
	\begin{equation}
		\partial_s \bar{G}(\theta,s) = \frac{\lambda_2(1+D\phi_1(\theta)) - \lambda_1 D\phi_2(\theta)}{\big(1+\lambda_1 s + D\phi_1(\theta)\big)^2} = \frac{J(\theta)}{\big(1+\lambda_1 s + D\phi_1(\theta)\big)^2}.
	\end{equation}
	{Let $\bar G_n$ be the fiber component of $\bar{\mathcal F}^n$
on a neighborhood of $\bar\Gamma_p$ where this iterate is defined.
On the invariant section, the chain rule gives}
	\begin{equation}
		A_n(\theta) := \partial_s \bar{G}_n(\theta, \bar{p}(\theta)) = \frac{\prod_{i=0}^{n-1} J(\bar{g}^i(\theta))}{P_n(\theta)^2} \le \frac{\lambda^n}{P_n(\theta)^2}.
	\end{equation}
	
	\subsection{Proof of Theorem~\ref{thm:C_infty_smoothness}}
	We complete the proof by induction on the regularity of $\psi$. Assume that $\psi \in C^j(\mathbb{T}, \mathbb{R})$ for some $1 \le j \le k-1$. We prove that $\psi \in C^{j+1}(\mathbb{T}, \mathbb{R})$, which is equivalent to showing that the slope section $\bar{p}$ is of class $C^j$.
	
	Under the induction hypothesis, $\bar g$ is a $C^j$ diffeomorphism,
and $\bar G|_V$ is of class $C^j$. Let
$m_n:=\inf_{\theta\in\mathbb T}P_n(\theta)$. The inverse derivative identity gives
\[
 D(\bar g^{-n})(\theta)
 =\frac{1}{P_n(\bar g^{-n}(\theta))},\qquad
 \kappa_n:=\sup_{\theta\in\mathbb T}|D\bar g^{-n}(\theta)|=m_n^{-1}.
\]
Using \eqref{eq:derivative_exponential_bound}, we obtain
	\[
	\left( \sup_{\theta \in \mathbb{T}} A_n(\theta) \right) \kappa_n^j \le \frac{\lambda^n}{m_n^2} \frac{1}{m_n^j} = \frac{\lambda^n}{m_n^{j+2}} \le \lambda^n e^{(j+2)\varepsilon n}.
	\]
	Since $\lambda \in (0,1)$, we choose $\varepsilon \in \big(0, \frac{-\log \lambda}{k+1}\big)$, so that $q := \lambda e^{(k+1)\varepsilon} < 1$. For this $\varepsilon$, choose an integer $N \ge N_0(\varepsilon)$. By the degree-one identity applied to the $N$-th iterate, $\int_{\mathbb{T}} P_N(\theta)\,d\theta = 1$, which implies $m_N = \inf_{\theta \in \mathbb{T}} P_N(\theta) \le 1$, and consequently $\kappa_N = m_N^{-1} \ge 1$. Then, on the invariant section $\bar{\Gamma}_p$, for the induction index $j \le k-1$, we have:
	\begin{equation}
		\left( \sup_{\theta \in \mathbb{T}} A_N(\theta) \right) \kappa_N^j \le \lambda^N e^{(j+2)\varepsilon N} \le q^N < 1.
	\end{equation}
	
	Define $V_N := \big\{ z \in V \mid \bar{\mathcal{F}}^{\ell}(z) \in V \text{ for every } 0 \le \ell \le N-1 \big\}$. Since $\bar{\Gamma}_p$ is compact and invariant, $V_N$ is an open neighborhood of $\bar{\Gamma}_p$, and the $N$-step map $\bar{\mathcal{F}}^N = (\bar{g}^N, \bar{G}_N)$ is well-defined and of class $C^j$ on $V_N$.
	
	Set $K_0 := \sup_{\theta \in \mathbb{T}} |\partial_s \bar{G}_N(\theta, \bar{p}(\theta))| = \sup_{\theta \in \mathbb{T}} A_N(\theta)$. Since $K_0 \kappa_N^j \le q^N < 1$, choose $\eta > 0$ such that $(K_0 + \eta)\kappa_N^j < 1$. By the continuity of $\partial_s \bar{G}_N$ near the compact section $\bar{\Gamma}_p$, there exists $\delta > 0$ such that the compact fiberwise tube
	\[
	\mathcal{K}_\delta := \big\{ (\theta,s) \in \mathbb{T} \times \mathbb{R} \;\big|\; |s - \bar{p}(\theta)| \le \delta \big\}
	\]
	is contained in $V_N$, and $|\partial_s \bar{G}_N(\theta,s)| \le K_0 + \eta$ whenever $(\theta, s) \in \mathcal{K}_\delta$. Consequently,
	\[
	K := \sup_{(\theta,s) \in \mathcal{K}_\delta} |\partial_s \bar{G}_N(\theta,s)| \le K_0 + \eta
	\]
	satisfies:
	\[
	K \kappa_N^j < 1.
	\]
	Since $\kappa_N \ge 1$, this also implies $K < 1$.
	
	Set $\tau := \delta/2$. By smooth approximation, choose $p_0 \in C^\infty(\mathbb{T}, \mathbb{R})$ such that
	\[
	\|\bar{p} - p_0\|_{C^0} < \frac{\delta}{4} = \frac{\tau}{2}.
	\]
	Define the smooth tube $W := \{ (\theta,s) \mid |s - p_0(\theta)| < \tau \}$. For every $(\theta, s) \in W$, the triangle inequality yields:
	\[
	|s - \bar{p}(\theta)| \le |s - p_0(\theta)| + |p_0(\theta) - \bar{p}(\theta)| < \frac{\delta}{2} + \frac{\delta}{4} = \frac{3\delta}{4} < \delta.
	\]
	These estimates show that $W \subset \mathcal{K}_\delta \subset V_N$.

Choose a smooth function $\varsigma \colon \mathbb{R} \to (-\tau, \tau)$ satisfying $\varsigma(t) = t$ for $|t| \le \tau/2$ and $|\varsigma'(t)| \le 1$ for all $t \in \mathbb{R}$. Define
	\[
	\widehat{G}_N(\theta,s) := \bar{G}_N\big(\theta, \, p_0(\theta) + \varsigma(s - p_0(\theta))\big), \qquad \widehat{\mathcal{F}}_N(\theta,s) := \big(\bar{g}^N(\theta), \, \widehat{G}_N(\theta,s)\big).
	\]
	Since $|p_0(\theta) + \varsigma(s - p_0(\theta)) - p_0(\theta)| = |\varsigma(s - p_0(\theta))| < \tau$, the argument of $\bar{G}_N$ belongs to $W \subset V_N$. Thus $\widehat{\mathcal{F}}_N$ is globally well-defined on $\mathbb{T} \times \mathbb{R}$ and of class $C^j$.
	
	Let $\sigma := \tau/2 - \|\bar{p} - p_0\|_{C^0} > 0$. For any $(\theta, s)$ such that $|s - \bar{p}(\theta)| < \sigma$,
	\[
	|s - p_0(\theta)| \le |s - \bar{p}(\theta)| + |\bar{p}(\theta) - p_0(\theta)| < \sigma + \|\bar{p} - p_0\|_{C^0} = \frac{\tau}{2}.
	\]
	Thus $\varsigma(s - p_0(\theta)) = s - p_0(\theta)$, so $\widehat{G}_N(\theta,s) = \bar{G}_N(\theta,s)$ on a neighborhood of the invariant section $\bar{\Gamma}_p$. In particular, $\widehat{G}_N(\theta, \bar{p}(\theta)) = \bar{p}(\bar{g}^N(\theta))$, so $\bar{p}$ is a continuous invariant section of $\widehat{\mathcal{F}}_N$.
	
	The chain rule gives
	\[
	|\partial_s \widehat{G}_N(\theta,s)| \le |\partial_s \bar{G}_N\bigl(\theta,p_0(\theta)+\varsigma(s-p_0(\theta))\bigr)| \cdot |\varsigma'(s - p_0(\theta))| \le K.
	\]
	Thus $K$ is a global fiber Lipschitz bound for $\widehat{G}_N$. Since $K < 1$ and $K \kappa_N^j < 1$, the map $\widehat{\mathcal{F}}_N$ is a global $j$-fiber contraction on $\mathbb{T} \times \mathbb{R}$.

Lemma~\ref{lem:C^r section}, with $r=j$ and base map
$\bar h=\bar g^N$, gives $\bar p\in C^j$, hence $\psi\in C^{j+1}$.
Induction proves $\psi\in C^k$.
If $F_{\lambda,\alpha}^{\phi}$ is $C^\infty$, applying this result
for every finite $k\ge2$ gives $\psi\in C^\infty$.
This proves Theorem~\ref{thm:C_infty_smoothness}.

\subsection{A Lipschitz invariant graph that is not $C^1$}
\label{sec:irrational_Lip_not_C1}

The $C^1$ hypothesis in Theorem~\ref{thm:C_infty_smoothness} cannot be
replaced by Lipschitz continuity, even for analytic maps with a single
sinusoidal perturbation.
\begin{Proposition}\label{prop:irrational_Lip_not_C1}
Fix $\lambda\in(0,1)$ and $\omega\in\mathbb R\setminus\mathbb Q$.
For $t\ge0$ and $\alpha\in\mathbb R$, consider
\begin{equation}\label{eq:critical_sine_family}
 F_{t,\alpha}(x,y)
 :=\bigl(x+\alpha+\lambda y+\phi_t(x),\,
          \lambda y+\phi_t(x)\bigr),
 \qquad \phi_t(x):=\frac{t}{2\pi}\sin(2\pi x).
\end{equation}
There exist $t_*\in(0,1+\lambda)$ and $\alpha_*\in\mathbb R$ such that
$F_{t_*,\alpha_*}$ admits an invariant graph
\[
 \Gamma_*:=\{(x,\psi_*(x)):x\in\mathbb R\},
 \qquad
 \psi_*\in C^{0,1}(\mathbb T,\mathbb R)\setminus C^1(\mathbb T,\mathbb R).
\]
Its induced degree-one lift
\[
 g_*(x):=x+\alpha_*+\lambda\psi_*(x)+\phi_{t_*}(x)
\]
satisfies
\[
 \lim_{n\to\infty}\frac{g_*^n(x)-x}{n}=\omega.
\]
\end{Proposition}

\begin{Remark}\label{rem:Hall_irrational_graphs}
Hall~\cite[Theorem~3]{Hall83} constructed a $C^\infty$ diffeomorphism of
an annulus with an attracting Lipschitz invariant circle carrying Denjoy
dynamics, that is, with no periodic or dense orbits. In his final
Remarks~(1) and~(2), he asked whether such examples exist for arbitrary
prescribed irrational rotation numbers and whether the ambient
diffeomorphism can be real analytic~\cite[p.~117]{Hall83}.
Proposition~\ref{prop:irrational_Lip_not_C1} provides analytic
dissipative twist maps with Lipschitz, non-$C^1$ invariant graphs for
every prescribed irrational rotation number. It thus establishes the
corresponding statement about the regularity of invariant graphs.
The proposition does not determine whether the induced dynamics has
wandering intervals.
\end{Remark}

The following parameter dependent statement follows from the $C^1$ persistence
theorem for normally hyperbolic invariant manifolds
\cite[Theorem~4.1]{HPS77}.
\begin{Lemma}[Local persistence with parameters]
\label{lem:local_parameter_persistence}
Let $U\subset\mathbb R^\ell$ be open, let $q_0\in U$, and let
$(q,\zeta)\mapsto f_q(\zeta)$ be a jointly $C^1$ family of
diffeomorphisms of $\mathbb T^d\times\mathbb R^d$, where
$\ell,d\ge1$. Suppose $f_{q_0}$ preserves a $C^1$ graph
$\bar\Gamma_0=\{(\theta,\bar\psi_0(\theta)):\theta\in\mathbb T^d\}$
and admits a continuous invariant splitting
\[
 T(\mathbb T^d\times\mathbb R^d)|_{\bar\Gamma_0}
 =T\bar\Gamma_0\oplus E^s.
\]
Assume that, for some $C\ge1$ and $0<\nu<1\le\mu$ with
$\nu\mu<1$, and every $n\ge0$,
\[
 \sup_{\bar\Gamma_0}\|Df_{q_0}^n|_{E^s}\|\le C\nu^n,
 \qquad
 \sup_{\bar\Gamma_0}\|Df_{q_0}^{\pm n}|_{T\bar\Gamma_0}\|
 \le C\mu^n.
\]
Then there is a neighborhood $V\subset U$ of $q_0$ and a jointly
$C^1$ function $\bar\psi:V\times\mathbb T^d\to\mathbb R^d$ such that
$\bar\psi(q_0,\cdot)=\bar\psi_0$ and
$\{(\theta,\bar\psi(q,\theta)):\theta\in\mathbb T^d\}$ is
$f_q$-invariant for every $q\in V$.
{\color{black}
After shrinking $V$, there exists $\varepsilon>0$ such that, for every
fixed $q\in V$, any $f_q$-invariant graph of a function
$\bar\varphi\in C^1(\mathbb T^d,\mathbb R^d)$ satisfying
$\|\bar\varphi-\bar\psi_0\|_{C^1}<\varepsilon$ coincides with the
graph of $\bar\psi(q,\cdot)$.
}
\end{Lemma}

\begin{proof}
Let $\mathcal P:=\mathbb T^\ell$. Choose a smooth map
$\chi:\mathcal P\to\mathbb R^\ell$ that agrees with local coordinates
near $0$ and vanishes outside a coordinate neighborhood. For small
$\sigma>0$, all parameters $q_0+\sigma\chi(p)$ belong to $U$, and
\[
 \mathcal F_\sigma(p,\zeta)
 :=\bigl(p,f_{q_0+\sigma\chi(p)}(\zeta)\bigr)
\]
is a $C^1$ diffeomorphism. Joint $C^1$ regularity gives
$\mathcal F_\sigma\to\mathcal F_0$, where $\mathcal F_0(p,\zeta)=(p,f_{q_0}(\zeta))$, in $C^1$ on a fixed compact
neighborhood of $\mathcal P\times\bar\Gamma_0$. In particular, the
parameter derivative of its second component is
$\sigma D_qf_{q_0+\sigma\chi(p)}(\zeta)D\chi(p)$ and tends uniformly
to zero. For $\mathcal F_0$, the tangent and stable bundles of the
product are $T\mathcal P\oplus T\bar\Gamma_0$ and $\{0\}\oplus E^s$.
The parameter directions have identity dynamics. Since $\mu\ge1$,
the same estimates hold for this product, after increasing $C$ if
necessary. Thus $\nu\mu<1$ makes it normally attracting.

The $C^1$ persistence theorem gives a nearby $C^1$ invariant manifold
that is $C^1$-close to the product. Its projection to
$\mathcal P\times\mathbb T^d$ is a diffeomorphism: under a $C^1$
parametrization by the original compact product, this projection is
$C^1$-close to the identity. The invariant manifold has the form
\[
 \{(p,\theta,\Psi_\sigma(p,\theta)):
 (p,\theta)\in\mathcal P\times\mathbb T^d\},
 \qquad \Psi_\sigma\in C^1.
\]
Since $\mathcal F_\sigma$ fixes $p$, each slice is invariant under
$f_{q_0+\sigma\chi(p)}$. Fix a sufficiently small $\sigma>0$.
In the coordinates where $\chi(p)=p$, set
\[
 \bar\psi(q,\theta)
 :=\Psi_\sigma\bigl((q-q_0)/\sigma,\theta\bigr)
\]
for $q$ near $q_0$. This function is jointly $C^1$.
{\color{black}
We prove the identification at $q_0$ and local uniqueness directly.
Set
\[
 R_0(\theta,y):=y-\bar\psi_0(\theta).
\]
For $z\in\bar\Gamma_0$, let $\Pi_z^s$ denote the projection onto
$E_z^s$ along $T_z\bar\Gamma_0$. The differential $DR_0$ vanishes
on $T\bar\Gamma_0$. Invariance of the splitting gives, for
$u\in\mathbb R^d$ and $n\ge0$,
\[
 D_y(R_0\circ f_{q_0}^n)(z)u
 =DR_0(f_{q_0}^n(z))\,Df_{q_0}^n(z)\Pi_z^s(0,u).
\]
Indeed, the tangent component of $(0,u)$ is mapped into the tangent
space of the graph and is annihilated by $DR_0$. Compactness bounds
$DR_0$ and $\Pi^s$ uniformly on the graph, so the stable estimate implies
\[
 \sup_{z\in\bar\Gamma_0}
 \|D_y(R_0\circ f_{q_0}^n)(z)\|\le K\nu^n
\]
for some $K$ independent of $n$. Choose $N\ge1$ with
$K\nu^N<1/4$. By continuity, there is $\eta_0>0$ such that
\[
 \|D_y(R_0\circ f_{q_0}^N)(z)\|\le\tfrac12
 \qquad\text{whenever }|R_0(z)|<\eta_0.
\]
For $z=(\theta,y)$ in this tube, the vertical segment from
$(\theta,\bar\psi_0(\theta))$ to $z$ stays in the tube.
Since $R_0\circ f_{q_0}^N$ vanishes on $\bar\Gamma_0$, integration
along this segment yields
\[
 |R_0(f_{q_0}^N(z))|\le\tfrac12|R_0(z)|.
\]
If an invariant graph $\bar\Gamma$ is contained in the tube, then
$f_{q_0}^N(\bar\Gamma)=\bar\Gamma$. Its compactness gives
\[
 M:=\max_{z\in\bar\Gamma}|R_0(z)|
   =\max_{z\in\bar\Gamma}|R_0(f_{q_0}^N(z))|
   \le\tfrac12 M.
\]
Thus $M=0$, and $\bar\Gamma=\bar\Gamma_0$. For sufficiently small
$\sigma$, the slice of the persistent manifold at $p=0$ lies in this
tube. This proves $\bar\psi(q_0,\cdot)=\bar\psi_0$.

Now put $R_q(\theta,y):=y-\bar\psi(q,\theta)$ and keep $N$ fixed.
Joint $C^1$ dependence and compactness allow us to shrink $V$ and
choose $\eta>0$ so that
\[
 \|D_y(R_q\circ f_q^N)(z)\|\le\tfrac34
 \qquad(q\in V,\ |R_q(z)|<\eta).
\]
The same integration along vertical segments gives
$|R_q(f_q^N(z))|\le\tfrac34|R_q(z)|$ in this tube.
The maximum argument above, now for the fixed map $f_q$, proves
uniqueness among invariant graphs contained in the tube. Finally,
shrink $V$ so that
$\|\bar\psi(q,\cdot)-\bar\psi_0\|_{C^0}<\eta/2$ and take
$\varepsilon=\eta/2$. If
$\|\bar\varphi-\bar\psi_0\|_{C^1}<\varepsilon$, its graph lies in
$|R_q|<\eta$, which proves the stated uniqueness.
}
\end{proof}

\begin{proof}[{Proof of Proposition~\ref{prop:irrational_Lip_not_C1}}]

Fix $\lambda$ and $\omega$ as in the proposition, and use the family
\eqref{eq:critical_sine_family}. We write $f_{t,\alpha}$ for its
projection to the cylinder and use the lift rotation number
$\rho(g)$ from Definition~\ref{def:rotation_number}.
Since $\phi_t=D\Phi_t$, with
$\Phi_t(x)=-t\cos(2\pi x)/(4\pi^2)$, this is the ECS family
\eqref{eq:conformal_symplectic_twist_map} in dimension one; its maps
are real-analytic diffeomorphisms with determinant $\lambda$ and
twist derivative $\partial_y(\pi_xF_{t,\alpha})=\lambda>0$.
The same formula defines the family for negative $t$ when an open
parameter neighborhood is needed in Step~3.

\medskip
\noindent\textbf{Step 1: Uniform estimates for invariant graphs.}
{Let $t\ge0$ and let $\psi\in C^0(\mathbb T,\mathbb R)$
define an invariant graph. Set}
\[
 g(x):=x+\alpha+\lambda\psi(x)+\phi_t(x).
\]
{By Proposition~\ref{pro:Herman}, $g$ is a strictly
increasing degree-one lift. The invariance equation gives}
\begin{equation}\label{eq:critical_inverse_identity}
 \psi(g(x))=\lambda\psi(x)+\phi_t(x)=g(x)-x-\alpha,
 \qquad g^{-1}(x)=x-\alpha-\psi(x).
\end{equation}
{The Herman--Mather identity \eqref{eq:Herman_formula}
reduces to}
\begin{equation}\label{eq:critical_HM_identity}
 g(x)+\lambda g^{-1}(x)
 =(1+\lambda)x+(1-\lambda)\alpha+\phi_t(x).
\end{equation}
Both terms on the left are increasing. If $t>1+\lambda$, the derivative
of the right-hand side is negative on a nonempty interval, which is
impossible. Thus any such graph has $t\le1+\lambda$.

Set $M_t:=1+\lambda+t$. For $u<v$, subtracting
\eqref{eq:critical_HM_identity} at $u$ from the same identity at $v$
cancels the constant term $(1-\lambda)\alpha$ and gives
\begin{align*}
 &g(v)-g(u)+\lambda\bigl(g^{-1}(v)-g^{-1}(u)\bigr)\\
 &\qquad=(1+\lambda)(v-u)+\phi_t(v)-\phi_t(u)\\
 &\qquad=(1+\lambda)(v-u)+\int_u^v t\cos(2\pi s)\,ds
 \le M_t(v-u).
\end{align*}
Since $g$ and $g^{-1}$ are strictly increasing and $\lambda>0$,
both terms on the left are positive. Each is bounded above by their
sum, so
\[
 0<g(v)-g(u)\le M_t(v-u),\qquad
 0<\lambda\bigl(g^{-1}(v)-g^{-1}(u)\bigr)\le M_t(v-u).
\]
Dividing by $v-u$ and taking the supremum over $u<v$ gives
\begin{equation}\label{eq:critical_biLip_bounds}
 {\|g\|_{\mathrm{Lip}}\le M_t,\qquad
 \|g^{-1}\|_{\mathrm{Lip}}\le M_t/\lambda.}
\end{equation}
Equation~\eqref{eq:critical_inverse_identity} implies that every
continuous invariant graph is Lipschitz, with
{$\|\psi\|_{\mathrm{Lip}}\le1+M_t/\lambda$}.
Also, surjectivity of $g$ and the first equality in
\eqref{eq:critical_inverse_identity} give
\begin{equation}\label{eq:critical_height_drift_bounds}
 \|\psi\|_{C^0}\le\frac{t}{2\pi(1-\lambda)},\qquad
 |\alpha-\rho(g)|\le\|\psi\|_{C^0}.
\end{equation}
For the second estimate, use $g(x)-x=\alpha+\psi(g(x))$, sum along an
orbit, and divide by the number of iterates.

The endpoint $t=1+\lambda$ is also excluded. Indeed,
\eqref{eq:critical_biLip_bounds} implies
\[
 g(v)-g(u)\ge\frac{\lambda}{M_t}(v-u),\qquad
 g^{-1}(v)-g^{-1}(u)\ge\frac{1}{M_t}(v-u).
\]
The increments of the right-hand side of
\eqref{eq:critical_HM_identity} are at least
$2\lambda(v-u)/M_t$. Its derivative cannot vanish, whereas for
$t=1+\lambda$ it vanishes at $x=1/2$.

\medskip
\noindent\textbf{Step 2: A maximal parameter with prescribed rotation number.}
Define
\[
 \mathcal S_\omega:=
 \left\{t\ge0:\begin{array}{l}
 \text{there exist }\alpha\in\mathbb R\text{ and a continuous invariant}\\
 \text{graph of }F_{t,\alpha}\text{ whose induced lift }g
 \text{ satisfies }\rho(g)=\omega
 \end{array}\right\}.
\]
At $t=0$, the graph $\psi=0$ with $\alpha=\omega$ shows that
$\mathcal S_\omega$ is nonempty. Step~1 gives
$\mathcal S_\omega\subset[0,1+\lambda)$.
Let $t_*:=\sup\mathcal S_\omega$ and choose $t_j\in\mathcal S_\omega$
with $t_j\to t_*$. Choose corresponding parameters $\alpha_j$ and
graphs $\psi_j$, and denote their induced lifts by $g_j$.
The bounds in Step~1 imply uniform boundedness and equicontinuity of
$\psi_j$, and boundedness of $\alpha_j$. After passing to a subsequence,
\[
 \alpha_j\to\alpha_*,\qquad
 \psi_j\to\psi_*\quad\text{uniformly},
\]
where $\psi_*$ is periodic and Lipschitz.
The lifts $g_j$ and their inverses converge uniformly to
\[
 g_*(x):=x+\alpha_*+\lambda\psi_*(x)+\phi_{t_*}(x),\qquad
 h_*(x):=x-\alpha_*-\psi_*(x),
\]
respectively. The uniform bounds in
\eqref{eq:critical_biLip_bounds} allow passage to the limit in
$g_j\circ g_j^{-1}=g_j^{-1}\circ g_j=\mathrm{Id}$, so
$h_*=g_*^{-1}$. Passing to the limit in the invariance equation shows
that $\psi_*$ is invariant under $F_{t_*,\alpha_*}$.
Continuity of the rotation number in the uniform topology
\cite{Her79} gives $\rho(g_*)=\omega$.
Thus $t_*\in\mathcal S_\omega$, and Step~1 shows that $t_*<1+\lambda$.

\medskip
\noindent\textbf{Step 3: Persistence if the limiting graph were $C^1$.}
Suppose, by contradiction, that $\psi_*\in C^1$.
Then $g_*\in C^1$, and the lower Lipschitz bound in Step~1 gives
\[
 Dg_*(x)\ge\frac{\lambda}{M_{t_*}}>0
 \qquad\text{for every }x\in\mathbb R.
\]
Thus $\bar g_*$ is an orientation-preserving $C^1$ circle
 diffeomorphism. Since $\rho(g_*)=\omega$ is irrational,
Theorem~\ref{thm:C_infty_smoothness} gives $\psi_*\in C^\infty$,
and the formula for $g_*$ gives $g_*\in C^\infty$.
We verify the normal hyperbolicity needed for persistence with
parameters. By \eqref{eq:derivative_exponential_bound} and the inverse
derivative identity, for every $\varepsilon>0$ there is
$C_\varepsilon\ge1$ such that, for all $n\ge0$ and $x\in\mathbb R$,
\begin{equation}\label{eq:critical_subexponential}
 C_\varepsilon^{-1}e^{-\varepsilon n}
 \le D(g_*^{\pm n})(x)
 \le C_\varepsilon e^{\varepsilon n}.
\end{equation}
Use coordinates $z=y-\psi_*(x)$, in which the graph is $z=0$, and denote
the map in these coordinates by $\widehat F_*$.
Writing $a(x):=Dg_*(x)$, the derivative along this graph is
\begin{equation}\label{eq:critical_triangular_derivative}
 D\widehat F_*(x,0)
 =\begin{pmatrix}a(x)&\lambda\\0&\lambda/a(x)\end{pmatrix}.
\end{equation}
Indeed, invariance makes the lower-left entry zero; the upper-right
entry is $\lambda$, and the determinant is $\lambda$.

{We construct the stable direction. Applying the notation of
Section~\ref{sec:circle_derivatives} to $g_*$, write $P_n(x)=D(g_*^n)(x)$ for $n\ge0$,
so $P_0=1$. Define}
\[
 \eta(x):=-\sum_{j=0}^{\infty}
 \frac{\lambda^{j+1}}{P_j(x)P_{j+1}(x)}.
\]
Choosing $\varepsilon>0$ with $\lambda e^{2\varepsilon}<1$ in
\eqref{eq:critical_subexponential} proves uniform convergence.
The resulting continuous periodic function satisfies
\[
 \eta(x)=\frac{\lambda}{a(x)^2}\eta(g_*(x))
          -\frac{\lambda}{a(x)}.
\]
Hence \eqref{eq:critical_triangular_derivative} maps
$(\eta(x),1)$ to $(\lambda/a(x))(\eta(g_*(x)),1)$.
This defines an invariant line bundle $E^s$ transverse to the tangent
bundle of the graph. Its $n$-step multiplier in this frame is
$\lambda^n/P_n(x)$, whereas the tangent multiplier is $P_n(x)$.
Since the frames and their inverses are uniformly bounded, on the
compact projected graph $\bar\Gamma_*$ one has
\[
 \sup_{\bar\Gamma_*}\|Df_{t_*,\alpha_*}^n|_{E^s}\|
 \le C_\varepsilon'\lambda^n e^{\varepsilon n},\qquad
 \sup_{\bar\Gamma_*}\|Df_{t_*,\alpha_*}^{\pm n}|_{T\bar\Gamma_*}\|
 \le C_\varepsilon' e^{\varepsilon n}.
\]

Choose $\varepsilon>0$ with $\lambda e^{2\varepsilon}<1$.
These estimates satisfy Lemma~\ref{lem:local_parameter_persistence}
with $q=(t,\alpha)$, $q_0=(t_*,\alpha_*)$,
$\nu=\lambda e^\varepsilon$, and $\mu=e^\varepsilon$, since
$\nu\mu=\lambda e^{2\varepsilon}<1$.
Taking periodic lifts of the graphs given by the lemma yields
$\psi_{t,\alpha}$ jointly $C^1$ in $(t,\alpha,x)$, with
$\psi_{t_*,\alpha_*}=\psi_*$. Write $g_{t,\alpha}$ for their induced
lifts, so $g_{t_*,\alpha_*}=g_*$.
{
After shrinking the parameter neighborhood, we may take it to be an
open rectangle about $(t_*,\alpha_*)$.
Equation~\eqref{eq:critical_inverse_identity} gives
\[
 g_{t,\alpha}^{-1}(x)=x-\alpha-\psi_{t,\alpha}(x),
\]
so these inverses are jointly $C^1$ in $(t,\alpha,x)$ as well.
}

\medskip
\noindent\textbf{Step 4: Varying the drift while preserving the rotation number.}
Fix $t=t_*$ and differentiate
\eqref{eq:critical_HM_identity} along the local family at
$\alpha=\alpha_*$. With
\[
 v(x):=\left.\partial_\alpha g_{t_*,\alpha}(x)
       \right|_{\alpha=\alpha_*},\qquad h:=g_*^{-1},
\]
the identity $g_{t_*,\alpha}(g_{t_*,\alpha}^{-1}(x))=x$ gives
\[
 \left.\partial_\alpha g_{t_*,\alpha}^{-1}(x)
 \right|_{\alpha=\alpha_*}
 =-\frac{v(h(x))}{Dg_*(h(x))}=-Dh(x)v(h(x)).
\]
Thus
\begin{equation}\label{eq:critical_drift_equation}
 v(x)-\lambda Dh(x)v(h(x))=1-\lambda.
\end{equation}
Iterating this equation $N$ times yields
\[
 v(x)=(1-\lambda)\sum_{j=0}^{N-1}\lambda^j D(h^j)(x)
       +\lambda^N D(h^N)(x)v(h^N(x)).
\]
The function $v$ is continuous and periodic, hence bounded.
{
For $\varepsilon>0$ with $\lambda e^\varepsilon<1$,
\eqref{eq:critical_subexponential} gives
\[
 \sup_{x\in\mathbb R}
 \bigl|\lambda^N D(h^N)(x)v(h^N(x))\bigr|
 \le C_\varepsilon(\lambda e^\varepsilon)^N\|v\|_{C^0}
 \longrightarrow0.
\]
The same geometric bound proves uniform convergence of the series.
} We obtain
\begin{equation}\label{eq:critical_positive_response}
 v(x)=(1-\lambda)\sum_{j=0}^{\infty}
              \lambda^jD(g_*^{-j})(x)\ge1-\lambda>0.
\end{equation}
{
By joint $C^1$ dependence, periodicity, and
\eqref{eq:critical_positive_response}, we can choose $s>0$ so that
the segment $\{t_*\}\times[\alpha_*-s,\alpha_*+s]$ lies in the
parameter rectangle and
\[
 \partial_\alpha g_{t_*,\alpha}(x)\ge\frac{1-\lambda}{2}
 \qquad(x\in\mathbb R,\ |\alpha-\alpha_*|\le s).
\]
Set $\alpha_-:=\alpha_*-s$ and $\alpha_+:=\alpha_*+s$.
Integrating this derivative bound over $[\alpha_-,\alpha_*]$ and
$[\alpha_*,\alpha_+]$ shows that both
$g_*(x)-g_{t_*,\alpha_-}(x)$ and
$g_{t_*,\alpha_+}(x)-g_*(x)$ are at least $s(1-\lambda)/2>0$.
In particular,
}
\[
 g_{t_*,\alpha_-}(x)<g_*(x)<g_{t_*,\alpha_+}(x)
 \qquad\text{for all }x\in\mathbb R.
\]
Under the contradiction hypothesis $\psi_*\in C^1$,
Theorem~\ref{thm:C_infty_smoothness} gives
$\bar g_*\in\operatorname{Diff}_+^\infty(\mathbb T)$, where
$\operatorname{Diff}_+^\infty(\mathbb T)$ denotes the orientation-preserving
smooth circle diffeomorphisms. Its rotation number is irrational, so
Denjoy's theorem applies. There is an increasing degree-one
homeomorphism $H$ such that
$H\circ g_*\circ H^{-1}(x)=x+\omega$.
For brevity, set
\[
 \widetilde g_\pm:=H\circ g_{t_*,\alpha_\pm}\circ H^{-1}.
\]
Evaluate the strict inequalities above at $H^{-1}(x)$ and apply the
strictly increasing map $H$. The conjugacy identity then gives
\[
 \widetilde g_-(x)<x+\omega<\widetilde g_+(x)
 \qquad\text{for every }x\in\mathbb R.
\]
The maps $H$, $H^{-1}$, and $g_{t_*,\alpha_\pm}$ commute with integer
translations, so $\widetilde g_\pm(x+1)=\widetilde g_\pm(x)+1$.
Hence the positive continuous functions
$x+\omega-\widetilde g_-(x)$ and
$\widetilde g_+(x)-x-\omega$ are $1$-periodic. Their minima on
$[0,1]$ satisfy
\[
 \delta_-:=\min_{x\in[0,1]}
 \bigl(x+\omega-\widetilde g_-(x)\bigr)>0,\qquad
 \delta_+:=\min_{x\in[0,1]}
 \bigl(\widetilde g_+(x)-x-\omega\bigr)>0.
\]
Periodicity makes these lower bounds valid on all of $\mathbb R$.
Thus
\[
 H\circ g_{t_*,\alpha_-}\circ H^{-1}(x)
 \le x+\omega-\delta_-,\qquad
 H\circ g_{t_*,\alpha_+}\circ H^{-1}(x)
 \ge x+\omega+\delta_+.
\]
{
For each $k\ge0$, apply these bounds at
$\widetilde g_-^{\,k}(x)$ and $\widetilde g_+^{\,k}(x)$, respectively:
\[
 \widetilde g_-^{\,k+1}(x)-\widetilde g_-^{\,k}(x)
 \le\omega-\delta_-,\qquad
 \widetilde g_+^{\,k+1}(x)-\widetilde g_+^{\,k}(x)
 \ge\omega+\delta_+.
\]
Summing over $k=0,\ldots,n-1$ cancels the intermediate terms, since
\[
 \sum_{k=0}^{n-1}
 \bigl(\widetilde g_\pm^{\,k+1}(x)-\widetilde g_\pm^{\,k}(x)\bigr)
 =\widetilde g_\pm^{\,n}(x)-x.
\]
This gives, for every $n\ge1$,
}
\[
 \widetilde g_-^{\,n}(x)\le x+n(\omega-\delta_-),\qquad
 \widetilde g_+^{\,n}(x)\ge x+n(\omega+\delta_+).
\]
Moreover,
$\widetilde g_\pm^{\,n}(H(x))=H(g_{t_*,\alpha_\pm}^{\,n}(x))$.
Since $H-\mathrm{Id}$ is periodic and bounded, dividing the
displacements in this identity by $n$ and letting $n\to\infty$
shows that $\rho(\widetilde g_\pm)=\rho(g_{t_*,\alpha_\pm})$.
Consequently,
\[
 \rho(g_{t_*,\alpha_-})\le\omega-\delta_-<\omega
 <\omega+\delta_+\le\rho(g_{t_*,\alpha_+}),
\]
which proves
\begin{equation}\label{eq:critical_rotation_bracket}
 \rho(g_{t_*,\alpha_-})<\omega<\rho(g_{t_*,\alpha_+}).
\end{equation}
Continuity preserves these strict inequalities when $t>t_*$ is
sufficiently close to $t_*$. For such $t$, the family is defined for
every $\alpha\in[\alpha_-,\alpha_+]$. Applying the intermediate value
theorem to $\alpha\mapsto\rho(g_{t,\alpha})$ gives a parameter
$\alpha(t)$ with $\rho(g_{t,\alpha(t)})=\omega$.
This places $t>t_*$ in $\mathcal S_\omega$, contradicting the definition
of $t_*$.

We conclude that $\psi_*\notin C^1$.
Finally, at $t=0$ the invariance equation forces
$\|\psi\|_{C^0}\le\lambda\|\psi\|_{C^0}$, so the only invariant graph
is $\psi=0$. Thus the non-$C^1$ graph just obtained has $t_*>0$,
which completes the proof.
\end{proof}

{\color{black}
\begin{proof}[Proof of Theorem~\ref{thm:main_regularity}]
Part~\textup{(i)} follows from Theorem~\ref{thm:C_infty_smoothness},
using Remark~\ref{rem:det} for the uniform determinant bounds.
Part~\textup{(ii)} is Proposition~\ref{prop:irrational_Lip_not_C1},
with $\phi_{t_*}=D\Phi$ for the potential $\Phi$ stated in the theorem.
\end{proof}
}

 \vspace{2em}

\noindent\textbf{Data Availability Statement.}
The authors state that this manuscript has no associated data and there is no conflict of interest.

	 \vspace{2em}
		
 \noindent\textbf{AI declaration.}
The authors used OpenAI’s ChatGPT as an auxiliary tool for proofreading, improving the presentation of some arguments.
All mathematical arguments were independently verified by the authors, who take full
responsibility for the content.

\vspace{2em}


\end{document}